\documentclass{amsart}
\usepackage{amssymb}
\usepackage{amsmath}
\usepackage{amsfonts}
\usepackage{diagrams}

\newtheorem{theorem}{Theorem}[section]
\newtheorem{acknowledgement}[theorem]{Acknowledgement}

\newtheorem{conjecture}[theorem]{Conjecture}
\newtheorem{corollary}[theorem]{Corollary}

\newtheorem{definition}[theorem]{Definition}
\newtheorem{example}[theorem]{Example}

\newtheorem{lemma}[theorem]{Lemma}
\newtheorem{notation}[theorem]{Notation}

\newtheorem{proposition}[theorem]{Proposition}
\newtheorem{remark}[theorem]{Remark}

\input{tcilatex}

\begin{document}
\title[Cosheafification]{Cosheafification}
\author{Andrei V. Prasolov}
\address{Institute of Mathematics and Statistics\\
The University of Troms\o\ - The Arctic University of Norway\\
N-9037 Troms\o , Norway}
\email{andrei.prasolov@uit.no}
\urladdr{http://serre.mat-stat.uit.no/ansatte/andrei/Welcome.html}
\date{}

\begin{abstract}
It is proved that for any Grothendieck site $X$, there exists a coreflection
(called \emph{cosheafification}) from the category of precosheaves on $X$
with values in a category $\mathbf{K}$, to the full subcategory of
cosheaves, provided either $\mathbf{K}$ or $\mathbf{K}^{op}$ is locally
presentable. If $\mathbf{K}$ is cocomplete, such a coreflection is built
explicitly for the (pre)cosheaves with values in the category $\mathbf{Pro}%
\left( \mathbf{K}\right) $ of pro-objects in $\mathbf{K}$. In the case of
precosheaves on topological spaces, it is proved that any precosheaf with
values in $\mathbf{Pro}\left( \mathbf{K}\right) $ is \emph{smooth}, i.e. is
strongly locally isomorphic to a cosheaf. Constant cosheaves are
constructed, and there are established connections with shape theory.
\end{abstract}

\maketitle
\tableofcontents

\begin{acknowledgement}
The author wishes to express his gratitude to Professor Carles Casacuberta.
The idea of part (\ref{Th-Main-K-op-locally-presentable}) of Theorem \ref%
{Th-Main} belongs to him. The long discussions with him helped the author to
understand the importance of locally presentable and accessible categories.
\end{acknowledgement}

\setcounter{section}{-1}

\section{Introduction}

A \emph{presheaf} (\emph{precosheaf}) on a topological space $X$ with values
in a category $\mathbf{K}$ is just a contravariant (covariant) functor from
the category of open subsets of $X$ to $\mathbf{K}$, while a \emph{sheaf} (%
\emph{cosheaf}) is such a functor satisfying some extra conditions.
Therefore, the category of (pre)cosheaves with values in $\mathbf{K}$ is
dual to the category of (pre)sheaves with values in the opposite category $%
\mathbf{K}^{op}$.

While the theory of sheaves is well developed, and is covered by a plenty of
publications, the theory of cosheaves is represented much poorer. The main
reason for this is that \emph{cofiltered limits} are \emph{not} exact in the
\textquotedblleft usual\textquotedblright\ categories like sets, abelian
groups, rings, or modules. On the contrary, \emph{filtered colimits} are
exact in the above categories, which allows to construct rather rich
theories of sheaves with values in the \textquotedblleft
usual\textquotedblright\ categories. To sum up, the \textquotedblleft
usual\textquotedblright\ categories $\mathbf{K}$ are badly suited for
cosheaf theory. Dually, the categories $\mathbf{K}^{op}$ are badly suited
for sheaf theory.

The first step in building a suitable theory of cosheaves would be
constructing a \emph{cosheaf associated with a precosheaf} (simply: \emph{%
cosheafification}). As is shown in this paper (see Theorem \ref{Th-Main}),
it is possible in many situations, namely for precosheaves with values in an
arbitrary \emph{locally presentable category} (or a dual to such a
category). The class of locally presentable categories is huge \cite[Ch. 1,
4 and 5]{Adamek-Rosicky-1994-Locally-presentable-categories-MR1294136}. It
includes all varieties and quasi-varieties of many-sorted algebras, and
essentially algebraic categories \cite[Theorem 3.36]%
{Adamek-Rosicky-1994-Locally-presentable-categories-MR1294136} of \emph{%
partial algebras} like the category $\mathbf{Cat}$ of small categories, and
the category $\mathbf{Pos}$ of posets. Even the class of locally \emph{%
finitely} presentable categories is very large, and includes \cite[Corollary
3.7 and Theorem 3.24]%
{Adamek-Rosicky-1994-Locally-presentable-categories-MR1294136} all varieties
of many-sorted \emph{finitary} algebras like $\mathbf{Set}$, $\mathbf{Gr}$, $%
\mathbf{Ab}$, modules etc. and all quasi-varieties like the category $%
\mathbf{Gra}$ of graphs, the category of torsion-free abelian groups, or the
category $\Sigma $-$\mathbf{Rel}$ of finitary relations.

However, our purpose is to prepare a foundation for future \emph{homology}
and \emph{homotopy} theories of cosheaves (see Conjectures \ref%
{Conj-Satellites-H}, \ref{Conj-Satellites-Pi} and \ref{Conj-Etale} below).
Therefore, we need a more or less \emph{explicit} construction. Moreover, we
need a construction satisfying \emph{good exactness} properties. In \cite%
{Funk-1995-The-display-locale-of-a-cosheaf-MR1322801} the cosheafification
of precosheaves of sets on topological spaces is discussed. It is sketched
there \cite[Theorem 6.3]{Funk-1995-The-display-locale-of-a-cosheaf-MR1322801}
that on complete metric spaces, the cosheafification can be described
explicitly by using the so-called \textquotedblleft display space of a
precosheaf\textquotedblright . See Example \ref{Ex-Non-smooth-precosheaf}
and \ref{Converging-sequence}. The construction of \cite[Theorem 6.3]%
{Funk-1995-The-display-locale-of-a-cosheaf-MR1322801} works there, but
produces cosheaves that are hardly interesting for future applications. In 
\cite[Appendix B]{Woolf-2009-MR2591969} it is claimed that the display space
construction works for \emph{any} topological space and \emph{any} cosheaf
of sets on it. However, his Lemma B.3 contains essential errors, see \cite%
{Woolf-2015-Erratum-to-The-fundamental-category-of-a-stratified-space-MR3313639}%
. Anyway, even the construction from \cite%
{Funk-1995-The-display-locale-of-a-cosheaf-MR1322801} for complete metric
spaces is \emph{not an exact functor}, and therefore is \emph{not} suitable
for homology and homotopy theories of cosheaves.

In \cite{Bredon-Book-MR1481706} and \cite{Bredon-MR0226631}, it is assumed
(correctly, in our opinion!) that a suitable cosheafification of a
precosheaf should be \emph{locally isomorphic} to the precosheaf. This
notion is much stronger than a $\mathbf{K}$-local isomorphism (Definition %
\ref{Def-Local-isomorphism}). We call a local isomorphism in the sense of
Bredon a \emph{strong local isomorphism} (Definition \ref%
{Def-Local-isomorphism-Bredon}). Precosheaves that admit a \textquotedblleft
correct\textquotedblright\ cosheafification are called \emph{smooth}
(Definition \ref{Def-smooth}, \cite[Corollary VI.3.2 and Definition VI.3.4]%
{Bredon-Book-MR1481706}, or \cite[Corollary 3.5 and Definition 3.7]%
{Bredon-MR0226631}). It is not clear whether one has enough smooth
precosheaves for building a suitable theory of cosheaves (see Example \ref%
{Non-smooth-precosheaf}, \ref{Ex-Non-smooth-precosheaf} and \ref%
{Converging-sequence}). In fact, Glen E. Bredon back in 1968 was rather
pessimistic on the issue. See \cite{Bredon-MR0226631}, p. 2:
\textquotedblleft \emph{The most basic concept in sheaf theory is that of a
sheaf generated by a given presheaf. In categorical terminology this is the
concept of a reflector from presheaves to sheaves. We believe that there is
not much hope for the existence of a reflector from precosheaves to cosheaves%
}\textquotedblright . It seems that he was still pessimistic in 1997:
Chapter VI \textquotedblleft Cosheaves and \v{C}ech
homology\textquotedblright\ of his book \cite{Bredon-Book-MR1481706} is
almost identical to \cite{Bredon-MR0226631}.

On the contrary, our approach seems to have solved the problem. If one
allows (pre)cosheaves (defined on an \emph{arbitrary} small Grothendieck
site) to take values in a larger category, then the desired reflection (in
fact, \textbf{co}reflection) can be constructed. It follows from our
considerations in this paper, that the best candidate for such category is
the pro-category $\mathbf{Pro}\left( \mathbf{K}\right) $ (see Definition \ref%
{Def-Pro-C}) for an \emph{arbitrary cocomplete} category $\mathbf{K}$. Our
cosheafification is built like this (Definition \ref{Def-Plus-cosheaf}):%
\begin{equation*}
\mathcal{A}\longmapsto \mathcal{A}_{+}\longmapsto \mathcal{A}_{++}=\mathcal{A%
}_{\#}.
\end{equation*}%
We have succeeded because of the niceness of the category $\mathbf{Pro}%
\left( \mathbf{K}\right) $. For \textquotedblleft usual\textquotedblright\
precosheaves (with values in $\mathbf{K}$ ) the above two-step process does
not work. In \cite{Prasolov-smooth-cosheaves-MR2879363}, this approach was
developed for precosheaves with values in $\mathbf{Pro}\left( \mathbf{Set}%
\right) $ and $\mathbf{Pro}\left( \mathbf{Ab}\right) $, and part (\ref%
{Main-constant-Set}) and (\ref{Main-constant-Ab}) of Theorem \ref%
{Main-constant} were proved. In this paper, the two statements are proved
much easier, using a significantly more general part (\ref%
{Main-constant-General}) of Theorem \ref{Main-constant}.

\begin{remark}
An interesting attempt is made in \cite{Schneiders-MR885939} where the
author sketches a cosheaf theory on topological spaces with values in a
category $\mathbf{L}$, dual to an \textquotedblleft
elementary\textquotedblright\ category $\mathbf{L}^{op}$. He proposes a
candidate for such a category. Let $\alpha <\beta $ be two inaccessible
cardinals. Then $\mathbf{L}$ is the category $\mathbf{Pro}_{\beta }\left( 
\mathbf{Ab}_{\alpha }\right) $ of abelian pro-groups $\left( G_{j}\right)
_{j\in \mathbf{J}}$ such that $card\left( G_{j}\right) <\alpha $ and $%
card\left( Mor\left( \mathbf{J}\right) \right) <\beta $. However, our\
pro-category $\mathbf{Pro}\left( \mathbf{K}\right) $ cannot be used in the
cosheaf theory from \cite{Schneiders-MR885939} because the category $\left( 
\mathbf{Pro}\left( \mathbf{K}\right) \right) ^{op}$ is \textbf{not}
elementary.
\end{remark}

\begin{remark}
Another cosheaf theory on topological spaces was sketched in \cite%
{Sugiki-2001-33}: the (pre)cosheaves there take values in the category $%
\mathbf{Pro}\left( \mathbf{Mod}\left( k\right) \right) $ where $k$ is a
commutative quasi-noetherian \cite[Definition 2.25]%
{Prasolov-universal-coefficients-formula-2013-MR3095217} ring.
\end{remark}

The cosheafification we have constructed guarantees that our precosheaves
are always smooth (Corollary \ref{Corollary-smooth}). Moreover, in Theorem %
\ref{Our-cosheaves-vs-Bredon}, we give necessary and sufficient conditions
for smoothness of a precosheaf with values in an \textquotedblleft
old\textquotedblright\ category $\mathbf{K}$:\ it is smooth iff our
coreflection applied to that precosheaf produces a cosheaf which takes
values in that old category.

Another difficulty in cosheaf theory is the lack of suitable \emph{constant}
cosheaves. In \cite{Bredon-Book-MR1481706} and \cite{Bredon-MR0226631}, such
cosheaves are constructed only for locally connected spaces. See Examples %
\ref{p0-prime-not-cosheaf} and \ref{Converging-sequence}. In Theorem \ref%
{Main-constant}, constant cosheaves are constructed. It turns out that they
are closely connected to shape theory. Namely, the constant cosheaf $\left(
G\right) _{\#}$ with values in $\mathbf{Pro}\left( \mathbf{K}\right) $ is
isomorphic to the \emph{pro-homotopy} (Definition \ref{Pro-homotopy-groups})
cosheaf $G\otimes _{\mathbf{Set}}pro$-$\pi _{0}$ (in particular $\left( 
\mathbf{pt}\right) _{\#}%
\simeq%
pro$-$\pi _{0}$), while the constant cosheaf $\left( A\right) _{\#}$ with
values in $\mathbf{Pro}\left( \mathbf{Ab}\right) $ is isomorphic to the 
\emph{pro-homology} (Definition \ref{Pro-homology-groups}) cosheaf $pro$-$%
H_{0}\left( \_,A\right) $.

In future papers, we are planning to develop homology of cosheaves, i.e. to
study projective and flabby cosheaves, projective and flabby resolutions,
and to construct the left satellites 
\begin{equation*}
H_{n}\left( X,\mathcal{A}\right) :=L_{n}\Gamma \left( X,\mathcal{A}\right)
\end{equation*}%
of the global sections functor 
\begin{equation*}
H_{0}\left( X,\mathcal{A}\right) :=\Gamma \left( X,\mathcal{A}\right) .
\end{equation*}%
It is expected that deeper connections to shape theory will be discovered,
as is stated in the two Conjectures below:

\begin{conjecture}
\label{Conj-Satellites-H}On the site $NORM\left( X\right) $ (Example \ref%
{Site-NORM}), the left satellites of $H_{0}$ are naturally isomorphic to the
pro-homology:%
\begin{equation*}
H_{n}\left( X,A_{\#}\right) =H_{n}\left( X,pro\text{-}H_{0}\left(
\_,A\right) \right) 
\simeq%
pro\text{-}H_{n}\left( X,A\right) .
\end{equation*}%
If $X$ is Hausdorff paracompact, the above isomorphisms exist also for the
standard site $OPEN\left( X\right) $ (Example \ref{Site-TOP}).
\end{conjecture}

\begin{conjecture}
\label{Conj-Satellites-Pi}On the site $NORM\left( X\right) $, the \textbf{%
non-abelian} left satellites of $H_{0}$ are naturally isomorphic to the
pro-homotopy:%
\begin{eqnarray*}
H_{n}\left( X,S_{\#}\right) &=&H_{n}\left( X,S\times pro\text{-}\pi
_{0}\right) 
\simeq%
S\times pro\text{-}\pi _{n}\left( X\right) , \\
H_{n}\left( X,\left( \mathbf{pt}\right) _{\#}\right) &=&H_{n}\left( X,pro%
\text{-}\pi _{0}\right) 
\simeq%
pro\text{-}\pi _{n}\left( X\right) .
\end{eqnarray*}%
If $X$ is Hausdorff paracompact, the above isomorphisms exist also for the
standard site $OPEN\left( X\right) $.
\end{conjecture}

The main application (Theorem \ref{Main-constant}) deals with the case of
topological spaces (i.e. the site $OPEN\left( X\right) $). Our
constructions, however, are valid for general Grothendieck sites. The
constructions in (strong) shape theory use essentially \emph{normal}
coverings instead of general coverings, therefore dealing with the site $%
NORM\left( X\right) $ instead of the site $OPEN\left( X\right) $. It seems
that Theorem \ref{Main-constant} is valid also for the site $NORM\left(
X\right) $. Applying our machinery (from this paper and from future papers)
to the site $FINITE\left( X\right) $ (Example \ref{Site-FINITE}), we expect
to obtain results on homology of the Stone-%
\u{C}ech \ %
compactification $\beta \left( X\right) $. To deal with the equivariant
homology, one should apply the machinery to the equivariant site $%
OPEN_{G}\left( X\right) $ (Example \ref{Equivariant}).

It is not yet clear how to generalize the above Conjectures to \emph{strong
shape theory}. However, we have some ideas how to do that.

Other possible applications could be in \'{e}tale homotopy theory \cite%
{Artin-Mazur-MR883959} as is summarized in the following

\begin{conjecture}
\label{Conj-Etale}Let $X^{et}$ be the site from Example \ref{Site-ETALE}.

\begin{enumerate}
\item The left satellites of $H_{0}$ are naturally isomorphic to the \'{e}%
tale pro-homology:%
\begin{equation*}
H_{n}\left( X^{et},A_{\#}\right) 
\simeq%
H_{n}^{et}\left( X,A\right) .
\end{equation*}

\item The non-abelian left satellites of $H_{0}$ are naturally isomorphic to
the \'{e}tale pro-homotopy:%
\begin{equation*}
H_{n}\left( X^{et},\left( \mathbf{pt}\right) _{\#}\right) 
\simeq%
H_{n}\left( X^{et},\pi _{0}^{et}\right) 
\simeq%
\pi _{n}^{et}\left( X\right) .
\end{equation*}
\end{enumerate}
\end{conjecture}

\section{Preliminaries}

\subsection{Categories}

\begin{notation}
~

\begin{enumerate}
\item We shall denote \textbf{limits} (inverse/projective limits) by $%
\underleftarrow{\lim }$, and \textbf{colimits} (direct/inductive limits) by $%
\underrightarrow{\lim }$.

\item If $U$ is an object of a category $\mathbf{C}$, we shall usually write 
$U\in \mathbf{C}$ instead of $U\in Ob\left( \mathbf{C}\right) $.
\end{enumerate}
\end{notation}

\begin{definition}
A \textbf{diagram} in $\mathbf{C}$ is a functor%
\begin{equation*}
\mathcal{D}:\mathbf{I}\longrightarrow \mathbf{C}
\end{equation*}%
where $\mathbf{I}$ is a \textbf{small} category. A \textbf{cone}
(respectively \textbf{cocone}) of the diagram $\mathcal{D}$\ is a pair $%
\left( U,\varphi \right) $ where $U\in \mathbf{C}$, and $\varphi $ is a
morphism of functors $\varphi :U^{const}\rightarrow \mathcal{D}$
(respectively $\mathcal{D}\rightarrow U^{const}$). Here $U^{const}$ is the
constant functor:%
\begin{eqnarray*}
U^{const}\left( i\right) &=&U,i\in \mathbf{I}, \\
U^{const}\left( i\rightarrow j\right) &=&\mathbf{1}_{U}.
\end{eqnarray*}
\end{definition}

\begin{remark}
We will also consider functors $\mathbf{C\rightarrow D}$ where $\mathbf{C}$
is not small. However, such functors form a \textbf{quasi-category} $\mathbf{%
D}^{\mathbf{C}}$, because the morphisms $\mathbf{D}^{\mathbf{C}}\left( 
\mathcal{F},\mathcal{G}\right) $ form a \textbf{class}, but not in general a 
\textbf{set}.
\end{remark}

\begin{definition}
\label{Def-(co)complete}A category $\mathbf{C}$ is called \textbf{complete}
if it admits small limits, and \textbf{cocomplete} if it admits small
colimits.
\end{definition}

\begin{remark}
A complete category has a \textbf{terminal} object (a limit of an empty
diagram). A cocomplete category has an \textbf{initial} object (a colimit of
an empty diagram).
\end{remark}

\begin{definition}
A functor $\mathcal{F}:\mathbf{C}\rightarrow \mathbf{D}$ is called \textbf{%
left (right) exact} if it commutes with \textbf{finite} limits (colimits). $%
\mathcal{F}$ is called \textbf{exact} if it is both left and right exact.
\end{definition}

\begin{definition}
\label{Def-(co)reflective}A subcategory $\mathbf{C\subseteq D}$ is called 
\textbf{reflective} (respectively \textbf{coreflective}) iff the inclusion $%
\mathbf{C\hookrightarrow D}$ is a right (respectively left) adjoint. The
left (respectively right) adjoint $\mathbf{D\rightarrow C}$ is called a 
\textbf{reflection} (respectively \textbf{coreflection}).
\end{definition}

\begin{definition}
Given $U\in \mathbf{C}$, let%
\begin{equation*}
h_{U}:\mathbf{C}^{op}\longrightarrow \mathbf{Set},~h^{U}:\mathbf{C}%
\longrightarrow \mathbf{Set},
\end{equation*}%
be the following functors:%
\begin{equation*}
h_{U}\left( V\right) 
{:=}%
Hom_{\mathbf{C}}\left( V,U\right) ,~h^{U}\left( V\right) 
{:=}%
Hom_{\mathbf{C}}\left( U,V\right) .
\end{equation*}
\end{definition}

\begin{remark}
The functors%
\begin{equation*}
h_{?}:\mathbf{C}\longrightarrow \mathbf{Set}^{\mathbf{C}^{op}},~h^{?}:%
\mathbf{C}^{op}\longrightarrow \mathbf{Set}^{\mathbf{C}},
\end{equation*}%
are full embeddings, called the \textbf{Yoneda embeddings}.
\end{remark}

\begin{definition}
\label{Def-Comma-U}Let $U\in \mathbf{C}$. The \textbf{comma category} $%
\mathbf{C}_{U}$ is defined as follows:%
\begin{eqnarray*}
&&Ob\left( \mathbf{C}_{U}\right) 
{:=}%
\left\{ \left( V\rightarrow U\right) \in Hom_{\mathbf{C}}\left( V,U\right)
\right\} , \\
&&Hom_{\mathbf{C}_{U}}\left( \left( \alpha _{1}:V_{1}\rightarrow U\right)
,\left( \alpha _{2}:V_{2}\rightarrow U\right) \right) 
{:=}%
\left\{ \beta :V_{1}\rightarrow V_{2}~|~\alpha _{2}\circ \beta =\alpha
_{1}\right\} .
\end{eqnarray*}
\end{definition}

\begin{definition}
\label{Def-Comma-R}Let $\mathcal{F}\in \mathbf{Set}^{\mathbf{C}^{op}}$. The 
\textbf{comma category} $\mathbf{C}_{\mathcal{F}}$ is defined as follows:%
\begin{eqnarray*}
&&Ob\left( \mathbf{C}_{\mathcal{F}}\right) 
{:=}%
\left\{ \left( V,\alpha \right) ~|~V\in \mathbf{C},\alpha \in \mathcal{F}%
\left( V\right) \right\} , \\
&&Hom_{\mathbf{C}_{U}}\left( \left( V_{1},\alpha _{1}\right) ,\left(
V_{2},\alpha _{2}\right) \right) 
{:=}%
\left\{ \beta :V_{1}\rightarrow V_{2}~|~\mathcal{F}\left( \beta \right)
\left( \alpha _{2}\right) =\alpha _{1}\right\} .
\end{eqnarray*}
\end{definition}

\begin{remark}
The categories $\mathbf{C}_{U}$ and $\mathbf{C}_{h_{U}}$ are equivalent.
\end{remark}

\subsection{Locally presentable categories}

The main reference here is \cite[Chapter 1]%
{Adamek-Rosicky-1994-Locally-presentable-categories-MR1294136}. See \cite[%
Definitions 1.1, 1.9, 1.13, and 1.17]%
{Adamek-Rosicky-1994-Locally-presentable-categories-MR1294136}.

\begin{notation}
We denote by $\aleph _{0}$ the smallest infinite cardinal.
\end{notation}

\begin{definition}
\label{Def-Locally-presentable}Let $\lambda $ be a regular cardinal, and $%
\mathbf{C}$ be a category.

\begin{enumerate}
\item A poset is called $\lambda $-\textbf{directed} provided that every
subset of cardinality smaller than $\lambda $ has an upper bound. A diagram $%
\mathcal{D}:\mathbf{I\rightarrow C}$ where $\mathbf{I}$ is a $\lambda $%
-directed poset is called a $\lambda $-\textbf{directed diagram}. A poset or
a diagram is called \textbf{directed} if it is $\aleph _{0}$-directed.

\item An object $U$ of $\mathbf{C}$ is called $\lambda $-presentable
provided that 
\begin{equation*}
h^{U}=Hom_{\mathbf{C}}\left( U,\_\right) :\mathbf{C}\longrightarrow \mathbf{%
Set}
\end{equation*}%
preserves $\lambda $-directed colimits. $U$ is called \textbf{finitely
presentable} if it is $\aleph _{0}$-presentable.

\item $\mathbf{C}$ is called \textbf{locally} $\lambda $-\textbf{presentable}
provided that it is cocomplete, and has a \textbf{set} $A$ of $\lambda $%
-presentable objects such that every object is a $\lambda $-directed colimit
of objects from $A$. $\mathbf{C}$ is called \textbf{locally} \textbf{%
presentable} if it is locally $\lambda $-presentable for some regular
cardinal $\lambda $. $\mathbf{C}$ is called \textbf{locally finitely
presentable} if it is locally $\aleph _{0}$-presentable.
\end{enumerate}
\end{definition}

\begin{remark}
The notions above can be equivalently defined using more general $\lambda $-%
\textbf{filtered} diagrams: a small category $\mathbf{I}$ is called $\lambda 
$-filtered provided that each subcategory with less than $\lambda $
morphisms has a cocone in $\mathbf{I}$. This means that:

\begin{enumerate}
\item $\mathbf{I}$ is non-empty.

\item For each collection $I_{s},s\in S$, of less than $\lambda $ objects of 
$\mathbf{I}$ there exists an object $J$ and morphisms $f_{s}:I_{s}%
\rightarrow J,s\in S$, in $\mathbf{I}$.

\item For each collection $g_{s}:I_{1}\rightarrow I_{2},s\in S$, of less
than $\lambda $ morphisms in $\mathbf{I}$ there exists a morphism $%
f:I_{2}\rightarrow J$ in $\mathbf{I}$ with $f\circ g_{s}$ independent of $s$.
\end{enumerate}

A diagram $\mathcal{D}:\mathbf{I\rightarrow C}$ is called $\lambda $%
-filtered if $\mathbf{I}$ is a $\lambda $-filtered category.
\end{remark}

\begin{remark}
\label{Rem-Locally-presentable}See \cite[Remark 1.19]%
{Adamek-Rosicky-1994-Locally-presentable-categories-MR1294136}. A category
is locally $\lambda $-presentable iff the following two conditions are
satisfied:

\begin{enumerate}
\item Every object is a $\lambda $-directed (equivalently: $\lambda $%
-filtered) colimit of $\lambda $-presentable objects.

\item There exists, up to an isomorphism, only a set of $\lambda $%
-presentable objects.
\end{enumerate}

By $Pres_{\lambda }\mathbf{C}$ we will denote a \textbf{set} of
representatives for the isomorphism classes of $\lambda $-presentable
objects of $\mathbf{C}$.
\end{remark}

\subsection{Pro-objects}

The main reference is \cite[Chapter 6]{Kashiwara-Categories-MR2182076}.

\begin{definition}
A category $\mathbf{I}$ is called \textbf{filtered} if $\mathbf{I}$ is $%
\aleph _{0}$-filtered. A category $\mathbf{I}$ is called \textbf{cofiltered}
if $\mathbf{I}^{op}$ is filtered. A diagram $\mathcal{D}:\mathbf{%
I\rightarrow K}$ is called (co)filtered if $\mathbf{I}$ is a (co)filtered
category.
\end{definition}

\begin{remark}
In \cite{Kashiwara-Categories-MR2182076}, such categories and diagrams are
called \textbf{(co)filtrant}.
\end{remark}

\begin{definition}
\label{Def-Pro-C}Let $\mathbf{K}$ be a category. The pro-category $\mathbf{%
Pro}\left( \mathbf{K}\right) $ (see \cite[Definition 6.1.1]%
{Kashiwara-Categories-MR2182076}, \cite[Remark I.1.4]%
{Mardesic-Segal-MR676973}, or \cite[Appendix]{Artin-Mazur-MR883959}) is the
category $\mathbf{L}^{op}$ where $\mathbf{L}\subseteq \mathbf{Set}^{\mathbf{K%
}}$ is the full subcategory of functors that are filtered colimits of
representable functors, i.e. colimits of diagrams of the form%
\begin{equation*}
\mathbf{I}^{op}\overset{\mathcal{X}^{op}}{\longrightarrow }\mathbf{K}^{op}%
\overset{h^{?}}{\longrightarrow }\mathbf{Set}^{\mathbf{K}}
\end{equation*}%
where $\mathbf{I}$ is a cofiltered category, $\mathcal{X}:\mathbf{I}%
\rightarrow \mathbf{K}$ is a diagram, and $h^{?}$ is the second Yoneda
embedding. We will simply denote such diagrams by $\mathcal{X}=\left(
X_{i}\right) _{i\in \mathbf{I}}$.

Let two pro-objects be defined by the diagrams $\mathcal{X}=\left(
X_{i}\right) _{i\in \mathbf{I}}$ and $\mathcal{Y}=\left( Y_{j}\right) _{j\in 
\mathbf{J}}$. Then%
\begin{equation*}
Hom_{\mathbf{Pro}\left( \mathbf{C}\right) }\left( \mathcal{X},\mathcal{Y}%
\right) =\underleftarrow{\lim }_{j\in \mathbf{J}}\underrightarrow{\lim }%
_{i\in \mathbf{I}}Hom_{\mathbf{C}}\left( X_{i},Y_{j}\right) .
\end{equation*}
\end{definition}

\begin{remark}
$\mathbf{Pro}\left( \mathbf{K}\right) $ is indeed a category even though $%
\mathbf{Set}^{\mathbf{K}}$ is a quasi-category: $Hom_{\mathbf{Pro}\left( 
\mathbf{K}\right) }\left( \mathcal{X},\mathcal{Y}\right) $ is a \textbf{set}
for any $\mathcal{X}$ and $\mathcal{Y}$.
\end{remark}

\begin{remark}
\label{Rem-Trivial-pro-object}\label{Rem-Rudimentary}The category $\mathbf{K}
$ is a full subcategory of $\mathbf{Pro}\left( \mathbf{K}\right) $: any
object $X\in \mathbf{K}$ gives rise to a \textbf{rudimentary} pro-object%
\begin{equation*}
\left( \mathbf{\ast }\longmapsto X\right) \in \mathbf{Pro}\left( \mathbf{K}%
\right) \mathbf{.}
\end{equation*}
\end{remark}

\subsection{Pro-homotopy and pro-homology}

Let $\mathbf{Top}\ $be the category of topological spaces and continuous
mappings. There are following categories closely connected to $\mathbf{Top}$%
: the category $H\left( \mathbf{Top}\right) $ of homotopy types, the
category $\mathbf{Pro}\left( H\left( \mathbf{Top}\right) \right) $ of
pro-homotopy types, and the category $H\left( \mathbf{Pro}\left( \mathbf{Top}%
\right) \right) $ of homotopy types of pro-spaces. The latter category is
used in \emph{strong shape theory}. It is finer than the former which is
used in \emph{shape theory}. The pointed versions $\mathbf{Pro}\left(
H\left( \mathbf{Top}_{\ast }\right) \right) $ and $H\left( \mathbf{Pro}%
\left( \mathbf{Top}_{\ast }\right) \right) $ are defined similarly.

One of the most important tools in strong shape theory is a \emph{strong
expansion} (see \cite{Mardesic-MR1740831}, conditions (S1) and (S2) on p.
129). In this paper, it is sufficient to use a weaker notion: an $H\left( 
\mathbf{Top}\right) $-\emph{expansion} (\cite[\S I.4.1]%
{Mardesic-Segal-MR676973}, conditions (E1) and (E2)). Those two conditions
are equivalent to the following

\begin{definition}
\label{HTOP-extension}Let $X$ be a topological space. A morphism $%
X\rightarrow \left( Y_{j}\right) _{j\in \mathbf{I}}$ in $\mathbf{Pro}\left(
H\left( \mathbf{Top}\right) \right) $ is called an $H\left( \mathbf{Top}%
\right) $-expansion (or simply \textbf{expansion}) if for any polyhedron $P$
the following mapping%
\begin{equation*}
\underrightarrow{\lim }_{j}\left[ Y_{j},P\right] =\underrightarrow{\lim }%
_{j}Hom_{H\left( \mathbf{Top}\right) }\left( Y_{j},P\right) \longrightarrow
Hom_{H\left( \mathbf{Top}\right) }\left( X,P\right) =\left[ X,P\right]
\end{equation*}%
is bijective where $\left[ Z,P\right] $ is the set of homotopy classes of
continuous mappings from $Z$ to $P$.

An expansion is called \textbf{polyhedral} (or an $H\left( \mathbf{Pol}%
\right) $-expansion) if all $Y_{j}$ are polyhedra.
\end{definition}

\begin{remark}
The pointed version of this notion (an $H\left( \mathbf{Pol}_{\ast }\right) $%
-expansion) is defined similarly.
\end{remark}

\begin{definition}
\label{Def-Normal-covering}An open covering is called \textbf{normal} \cite[%
\S I.6.2]{Mardesic-Segal-MR676973}, iff there is a partition of unity
subordinated to it.
\end{definition}

\begin{remark}
Theorem 8 from \cite[App.1, \S 3.2]{Mardesic-Segal-MR676973}, shows that an $%
H\left( \mathbf{Pol}\right) $- or an $H\left( \mathbf{Pol}_{\ast }\right) $%
-expansion for $X$ can be constructed using nerves of normal (see Definition %
\ref{Def-Normal-covering}) open coverings of $X$.
\end{remark}

Pro-homotopy is defined in \cite[p. 121]{Mardesic-Segal-MR676973}:

\begin{definition}
\label{Pro-homotopy-groups}For a (pointed) topological space $X$, define its
pro-homotopy pro-sets%
\begin{equation*}
pro\text{-}\pi _{n}\left( X\right) 
{:=}%
\left( \pi _{n}\left( Y_{j}\right) \right) _{j\in \mathbf{J}}
\end{equation*}%
where $X\rightarrow \left( Y_{j}\right) _{j\in \mathbf{J}}$ is an $H\left( 
\mathbf{Pol}\right) $-expansion if $n=0$, and an $H\left( \mathbf{Pol}_{\ast
}\right) $-expansion if $n\geq 1$.
\end{definition}

\begin{remark}
Similar to the \textquotedblleft usual\textquotedblright\ algebraic
topology, $pro$-$\pi _{0}$ is a pro-set (an object of $\mathbf{Pro}\left( 
\mathbf{Set}\right) $), $pro$-$\pi _{1}$ is a pro-group (an object of $%
\mathbf{Pro}\left( \mathbf{Gr}\right) $), and $pro$-$\pi _{n}$ are abelian
pro-groups (objects of $\mathbf{Pro}\left( \mathbf{Ab}\right) $) for $n\geq
2 $.
\end{remark}

Pro-homology groups are defined in \cite[\S II.3.2]{Mardesic-Segal-MR676973}%
, as follows:

\begin{definition}
\label{Pro-homology-groups}For a topological space $X$, and an abelian group 
$A$, define its pro-homology groups as%
\begin{equation*}
pro\text{-}H_{n}\left( X,A\right) 
{:=}%
\left( H_{n}\left( Y_{j},A\right) \right) _{j\in \mathbf{J}}
\end{equation*}%
where $X\rightarrow \left( Y_{j}\right) _{j\in \mathbf{J}}$ is an $H\left( 
\mathbf{Pol}\right) $-expansion.
\end{definition}

\section{Main results}

The proofs of the statements from this section will be given in Section \ref%
{Section-Proofs}.

\subsection{General sites}

Let $X=\left( \mathbf{C}_{X},Cov\left( X\right) \right) $ be a small site
(Definition \ref{Def-Site}), and let $\mathbf{K}$ be a category. By $\mathbf{%
Pro}\left( \mathbf{K}\right) $ we denote the corresponding pro-category
(Definition \ref{Def-Pro-C}).

Let $\mathbf{pCS}\left( X,\mathbf{K}\right) =\mathbf{K}^{\mathbf{C}_{X}}$ be
the category of precosheaves on $X$ with values in $\mathbf{K}$, and let $%
\mathbf{CS}\left( X,\mathbf{K}\right) $ (if $\mathbf{K}$ is cocomplete) be
the full subcategory of cosheaves (Notation \ref{Not-(Pre)cosheaves}).

\begin{theorem}
\label{Th-Main}~

\begin{enumerate}
\item \label{Th-Main-K-locally-presentable}If $\mathbf{K}$ is locally
presentable (Definition \ref{Def-Locally-presentable}), then $\mathbf{CS}%
\left( X,\mathbf{K}\right) \subseteq \mathbf{pCS}\left( X,\mathbf{K}\right) $
is a coreflective subcategory.

\item \label{Th-Main-K-op-locally-presentable}If $\mathbf{K}^{op}$ is
locally presentable, then $\mathbf{CS}\left( X,\mathbf{K}\right) \subseteq 
\mathbf{pCS}\left( X,\mathbf{K}\right) $ is a coreflective subcategory.

\item \label{Th-Main-K-op-locally-finitely-presentable}If $\mathbf{K}^{op}$
is locally finitely presentable, then the coreflection $\mathbf{pCS}\left( X,%
\mathbf{K}\right) \rightarrow \mathbf{CS}\left( X,\mathbf{K}\right) $ is
given by%
\begin{equation*}
\mathcal{A}\longmapsto \left( \mathcal{A}\right) _{\#}^{\mathbf{K}}=\left( 
\mathcal{A}\right) _{++}^{\mathbf{K}}
\end{equation*}%
(see Definition \ref{Def-Plus-cosheaf}).

\item \label{Th-Main-Pro(K)-K-(co)complete}Assume $\mathbf{K}$ is
cocomplete, and $\mathcal{A}\in \mathbf{pCS}\left( X,\mathbf{K}\right) $.
Then:

\begin{enumerate}
\item \label{Th-Main-Pro(K)-K-(co)complete-cosheaf}$\mathcal{A}$ is
coseparated (a cosheaf) iff it is coseparated (a cosheaf) when considered as
a precosheaf with values in $\mathbf{Pro}\left( \mathbf{K}\right) $.

\item 
\begin{equation*}
\mathbf{CS}\left( X,\mathbf{Pro}\left( \mathbf{K}\right) \right) \subseteq 
\mathbf{pCS}\left( X,\mathbf{Pro}\left( \mathbf{K}\right) \right)
\end{equation*}%
is coreflective, and the coreflection 
\begin{equation*}
\mathbf{pCS}\left( X,\mathbf{Pro}\left( \mathbf{K}\right) \right)
\longrightarrow \mathbf{CS}\left( X,\mathbf{Pro}\left( \mathbf{K}\right)
\right)
\end{equation*}%
is given by%
\begin{equation*}
\mathcal{A}\longmapsto \left( \mathcal{A}\right) _{\#}^{\mathbf{Pro}\left( 
\mathbf{K}\right) }=\left( \mathcal{A}\right) _{++}^{\mathbf{Pro}\left( 
\mathbf{K}\right) }.
\end{equation*}
\end{enumerate}
\end{enumerate}
\end{theorem}

\begin{corollary}
Assume that either $\mathbf{K}$ or $\mathbf{K}^{op}$ is locally presentable.
Then%
\begin{equation*}
\mathbf{CS}\left( X,\mathbf{Pro}\left( \mathbf{K}\right) \right) \subseteq 
\mathbf{pCS}\left( X,\mathbf{Pro}\left( \mathbf{K}\right) \right)
\end{equation*}%
is coreflective, and the coreflection 
\begin{equation*}
\mathbf{pCS}\left( X,\mathbf{Pro}\left( \mathbf{K}\right) \right)
\longrightarrow \mathbf{CS}\left( X,\mathbf{Pro}\left( \mathbf{K}\right)
\right)
\end{equation*}%
is given by%
\begin{equation*}
\mathcal{A}\longmapsto \left( \mathcal{A}\right) _{\#}^{\mathbf{Pro}\left( 
\mathbf{K}\right) }=\left( \mathcal{A}\right) _{++}^{\mathbf{Pro}\left( 
\mathbf{K}\right) }.
\end{equation*}
\end{corollary}

\begin{proof}
If $\mathbf{K}$ is locally presentable, it is both complete and cocomplete 
\cite[Corollary 2.47]%
{Adamek-Rosicky-1994-Locally-presentable-categories-MR1294136}. If $\mathbf{K%
}^{op}$ is locally presentable, then again $\mathbf{K}$ is both complete and
cocomplete. The statement follows from Theorem \ref{Th-Main} (\ref%
{Th-Main-Pro(K)-K-(co)complete}).
\end{proof}

\subsection{Topological spaces}

Throughout this subsection, $X$ is a topological space considered as the
site $OPEN\left( X\right) $ (see Example \ref{Site-TOP} and Remark \ref%
{Denote-standard-site-simply}).

\begin{theorem}
\label{Main-local-iso-pro-K}Let $\mathbf{K}$ be a cocomplete category.

\begin{enumerate}
\item For any precosheaf $\mathcal{A}$ on $X$ with values in $\mathbf{Pro}%
\left( \mathbf{K}\right) $, $\left( \mathcal{A}\right) _{\#}^{\mathbf{Pro}%
\left( \mathbf{K}\right) }\rightarrow \mathcal{A}$ is a strong local
isomorphism (Definition \ref{local-isomorphism}).

\item Any strong local isomorphism $\mathcal{A}\rightarrow \mathcal{B}$
between cosheaves on $X$ with values in $\mathbf{Pro}\left( \mathbf{K}%
\right) $, is an isomorphism.

\item If $\mathcal{B}\rightarrow \mathcal{A}$ is a strong local isomorphism,
and $\mathcal{B}$ is a cosheaf, then the natural morphism $\mathcal{B}%
\rightarrow \left( \mathcal{A}\right) _{\#}^{\mathbf{Pro}\left( \mathbf{K}%
\right) }$ is an isomorphism.
\end{enumerate}
\end{theorem}

\begin{definition}
\label{Def-smooth}A precosheaf $\mathcal{A}$ is called \textbf{smooth} (\cite%
[Corollary VI.3.2 and Definition VI.3.4]{Bredon-Book-MR1481706}, or \cite[%
Corollary 3.5 and Definition 3.7]{Bredon-MR0226631}), iff there exist
precosheaves $\mathcal{B}$ and $\mathcal{B}^{\prime }$, a cosheaf $\mathcal{C%
}$, and strong local isomorphisms $\mathcal{A}\rightarrow \mathcal{B}%
\leftarrow \mathcal{C}$, or, equivalently, strong local isomorphisms $%
\mathcal{A}\leftarrow \mathcal{B}^{\prime }\rightarrow \mathcal{C}$.
\end{definition}

Corollary \ref{Corollary-smooth} below guarantees that all precosheaves with
values in $\mathbf{Pro}\left( \mathbf{K}\right) $ are smooth:

\begin{corollary}
\label{Corollary-smooth}Let $\mathbf{K}$ be a cocomplete category. Any
precosheaf with values in $\mathbf{Pro}\left( \mathbf{K}\right) $ is smooth.
\end{corollary}

\begin{proof}
Consider the diagram%
\begin{equation*}
\mathcal{A}\overset{\mathbf{1}_{\mathcal{A}}}{\longrightarrow }\mathcal{A}%
\longleftarrow \left( \mathcal{A}\right) _{\#}^{\mathbf{Pro}\left( \mathbf{K}%
\right) },
\end{equation*}
or the diagram 
\begin{equation*}
\mathcal{A}\longleftarrow \left( \mathcal{A}\right) _{\#}^{\mathbf{Pro}%
\left( \mathbf{K}\right) }\overset{\mathbf{1}_{\mathcal{A}}}{\longrightarrow 
}\left( \mathcal{A}\right) _{\#}^{\mathbf{Pro}\left( \mathbf{K}\right) }.
\end{equation*}
\end{proof}

The results on cosheaves and precosheaves with values in $\mathbf{Pro}\left( 
\mathbf{K}\right) $ can be applied to \textquotedblleft
usual\textquotedblright\ ones (like in \cite{Bredon-MR0226631} and \cite[%
Chapter VI]{Bredon-Book-MR1481706}), with values in $\mathbf{K}$, because $%
\mathbf{K}$ is a full subcategory of $\mathbf{Pro}\left( \mathbf{K}\right) $
(see Remark \ref{Rem-Trivial-pro-object}).

The connection between the two types of (pre)cosheaves can be summarized in
the following

\begin{theorem}
\label{Our-cosheaves-vs-Bredon}Let $\mathbf{K}$ be cocomplete, and let $%
\mathcal{A}$ be a precosheaf on a topological space $X$ with values in $%
\mathbf{K}$.

\begin{enumerate}
\item $\mathcal{A}$ is coseparated (a cosheaf) iff it is coseparated (a
cosheaf) when considered as a precosheaf with values in $\mathbf{Pro}\left( 
\mathbf{K}\right) $.

\item $\mathcal{A}$ is smooth iff $\left( \mathcal{A}\right) _{\#}^{\mathbf{%
Pro}\left( \mathbf{K}\right) }$ takes values in $\mathbf{K}$, i.e. $\left( 
\mathcal{A}\right) _{\#}^{\mathbf{Pro}\left( \mathbf{K}\right) }\left(
U\right) \in \mathbf{K}$ (in other words, $\left( \mathcal{A}\right) _{\#}^{%
\mathbf{Pro}\left( \mathbf{K}\right) }\left( U\right) $ is a \emph{%
rudimentary} pro-object, see Remark \ref{Rem-Trivial-pro-object}) for any
open subset $U\subseteq X$.
\end{enumerate}
\end{theorem}

We are now able to construct \emph{constant} cosheaves, and to establish
connections to shape theory.

\begin{theorem}
\label{Main-constant}Let $\mathbf{K}$ be a cocomplete category, and let $%
G\in \mathbf{K}$.

\begin{enumerate}
\item \label{Main-constant-General}The precosheaf%
\begin{equation*}
\mathcal{P}\left( U\right) :=G\otimes _{\mathbf{Set}}pro\text{-}\pi
_{0}\left( U\right)
\end{equation*}%
(Definition \ref{Def-Pro(Set)-Ten-Set-K}), where $pro$-$\pi _{0}$ is the
pro-homotopy functor from Definition \ref{Pro-homotopy-groups} (see also 
\cite[p. 121]{Mardesic-Segal-MR676973}), is a cosheaf. Let $G$ be the
constant precosheaf corresponding to $G$ (Definition \ref{constant}) on $X$
with values in $\mathbf{Pro}\left( \mathbf{K}\right) $. Then $\left(
G\right) _{\#}$ is naturally isomorphic to $\mathcal{P}$.

\item \label{Main-constant-Set}Let $\mathbf{K=Set}$. The precosheaf%
\begin{equation*}
\mathcal{Q}\left( U\right) :=G\times pro\text{-}\pi _{0}\left( U\right)
\end{equation*}
is a cosheaf. Let $G$ be the constant precosheaf corresponding to $G$ on $X$
with values in $\mathbf{Pro}\left( \mathbf{Set}\right) $. Then $\left(
G\right) _{\#}$ is naturally isomorphic to $\mathcal{Q}$.

\item \label{Main-constant-Ab}Let $\mathbf{K=Ab}$. The precosheaf%
\begin{equation*}
\mathcal{H}\left( U\right) :=pro\text{-}H_{0}\left( U,G\right)
\end{equation*}%
where $pro$-$H_{0}$ is the pro-homology functor from Definition \ref%
{Pro-homology-groups} (see also \cite[\S II.3.2]{Mardesic-Segal-MR676973}),
is a cosheaf. Let $G$ be the constant precosheaf corresponding to $G$
(Definition \ref{constant-AB}) on $X$ with values in $\mathbf{Pro}\left( 
\mathbf{Ab}\right) $. Then $\left( G\right) _{\#}$ is naturally isomorphic
to $\mathcal{H}$.
\end{enumerate}
\end{theorem}

\begin{corollary}
\label{One-point-constant}\label{Cor-One-point-constant}
\end{corollary}

\begin{enumerate}
\item $pro$-$\pi _{0}$ is a cosheaf.

\item $\left( \mathbf{pt}\right) _{\#}%
\simeq%
pro$-$\pi _{0}$ where $\mathbf{pt}$ is the one-point constant precosheaf.
\end{enumerate}

\begin{proof}
Put $G=\mathbf{pt}$ in Theorem \ref{Main-constant} (\ref{Main-constant-Set}).
\end{proof}

\section{Pairings}

\begin{definition}
Let $\mathbf{K}$ be a category. Assume that $\mathbf{K}$ is complete in (\ref%
{Def-Hom-Set-Z-G}) below, and cocomplete in (\ref{Def-G-ten-Set-Z}) below.
Given%
\begin{equation*}
G,H\in \mathbf{K,~}Z\in \mathbf{Set,}
\end{equation*}%
let us consider

\begin{enumerate}
\item 
\begin{equation*}
Hom_{\mathbf{K}}\left( G,H\right) \in \mathbf{Set;}
\end{equation*}

\item \label{Def-Hom-Set-Z-G}%
\begin{equation*}
Hom_{\mathbf{Set}}\left( Z,G\right) 
{:=}%
\dprod\limits_{Z}G\in \mathbf{K;}
\end{equation*}

\item \label{Def-G-ten-Set-Z}%
\begin{equation*}
G\otimes _{\mathbf{Set}}Z=Z\otimes _{\mathbf{Set}}G%
{:=}%
\dcoprod\limits_{Z}G\in \mathbf{K;}
\end{equation*}
\end{enumerate}
\end{definition}

\begin{remark}
The first two assignments are contravariant on the first argument and
covariant on the second argument:%
\begin{eqnarray*}
Hom_{\mathbf{K}}\left( \_,\_\right) &:&\mathbf{K}^{op}\mathbf{\times K}%
\longrightarrow \mathbf{Set,} \\
Hom_{\mathbf{Set}}\left( \_,\_\right) &:&\mathbf{Set}^{op}\mathbf{\times K}%
\longrightarrow \mathbf{K,}
\end{eqnarray*}%
while the third functor is covariant on both arguments:%
\begin{eqnarray*}
\_\otimes _{\mathbf{Set}}\_ &:&\mathbf{K\times Set}\longrightarrow \mathbf{K,%
} \\
\_\otimes _{\mathbf{Set}}\_ &:&\mathbf{Set\times K}\longrightarrow \mathbf{K.%
}
\end{eqnarray*}
\end{remark}

\begin{definition}
\label{Def-Pro(Set)-Ten-Set-K}Let $\mathcal{X}=\left( X_{i}\right) _{i\in 
\mathbf{I}}\in \mathbf{Pro}\left( \mathbf{Set}\right) $, and $Y\in \mathbf{K}
$. Define%
\begin{equation*}
\mathcal{X}\otimes _{\mathbf{Set}}Y=Y\otimes _{\mathbf{Set}}\mathcal{X}\in 
\mathbf{Pro}\left( \mathbf{K}\right)
\end{equation*}%
by%
\begin{equation*}
\mathcal{X}\otimes _{\mathbf{Set}}Y=\left( X_{i}\otimes _{\mathbf{Set}%
}Y\right) _{i\in \mathbf{I}}.
\end{equation*}
\end{definition}

\begin{definition}
Let $\mathbf{K}$ be a complete and cocomplete category, let%
\begin{equation*}
\mathcal{A}:\mathbf{C}\longrightarrow \mathbf{K,~}\mathcal{B}:\mathbf{C}%
\longrightarrow \mathbf{Set,}
\end{equation*}%
be functors, and let%
\begin{equation*}
G\in \mathbf{K,~}Z\in \mathbf{Set.}
\end{equation*}%
Then define the following functors:

\begin{enumerate}
\item 
\begin{equation*}
Hom_{\mathbf{K}}\left( G,\mathcal{A}\right) :\mathbf{C}\longrightarrow 
\mathbf{Set,}
\end{equation*}%
\begin{equation*}
Hom_{\mathbf{K}}\left( G,\mathcal{A}\right) \left( U\right) 
{:=}%
Hom_{\mathbf{K}}\left( G,\mathcal{A}\left( U\right) \right) ;
\end{equation*}

\item 
\begin{equation*}
Hom_{\mathbf{K}}\left( \mathcal{A},G\right) :\mathbf{C}^{op}\longrightarrow 
\mathbf{Set,}
\end{equation*}%
\begin{equation*}
Hom_{\mathbf{K}}\left( \mathcal{A},G\right) \left( U\right) 
{:=}%
Hom_{\mathbf{K}}\left( \mathcal{A}\left( U\right) ,G\right) ;
\end{equation*}

\item 
\begin{equation*}
Hom_{\mathbf{Set}}\left( \mathcal{B},G\right) :\mathbf{C}^{op}%
\longrightarrow \mathbf{K,}
\end{equation*}%
\begin{equation*}
Hom_{\mathbf{Set}}\left( \mathcal{B},G\right) 
{:=}%
Hom_{\mathbf{Set}}\left( \mathcal{B}\left( U\right) ,G\right) ;
\end{equation*}

\item 
\begin{equation*}
\mathcal{A}\otimes _{\mathbf{Set}}Z=Z\otimes _{\mathbf{Set}}\mathcal{A}:%
\mathbf{C}\longrightarrow \mathbf{K,}
\end{equation*}%
\begin{equation*}
\left( \mathcal{A}\otimes _{\mathbf{Set}}Z\right) \left( U\right) =\left(
Z\otimes _{\mathbf{Set}}\mathcal{A}\right) \left( U\right) 
{:=}%
Z\otimes _{\mathbf{Set}}\left( \mathcal{A}\left( U\right) \right) ;
\end{equation*}

\item 
\begin{equation*}
\mathcal{B}\otimes _{\mathbf{Set}}G=G\otimes _{\mathbf{Set}}\mathcal{B}:%
\mathbf{C}\longrightarrow \mathbf{K,}
\end{equation*}%
\begin{equation*}
\left( \mathcal{B}\otimes _{\mathbf{Set}}G\right) \left( U\right) =\left(
G\otimes _{\mathbf{Set}}\mathcal{B}\right) \left( U\right) 
{:=}%
G\otimes _{\mathbf{Set}}\left( \mathcal{B}\left( U\right) \right) .
\end{equation*}
\end{enumerate}
\end{definition}

\begin{remark}
The first three assignments are contravariant on the first argument and
covariant on the second argument:%
\begin{eqnarray*}
Hom_{\mathbf{K}}\left( \_,\_\right) &:&\mathbf{K}^{op}\times \mathbf{K}^{%
\mathbf{C}}\longrightarrow \mathbf{Set}^{\mathbf{C}}\mathbf{,} \\
Hom_{\mathbf{K}}\left( \_,\_\right) &:&\left( \mathbf{K}^{\mathbf{C}}\right)
^{op}\mathbf{\times K}\longrightarrow \mathbf{Set}^{\mathbf{C}^{op}}\mathbf{,%
} \\
Hom_{\mathbf{Set}}\left( \_,\_\right) &:&\left( \mathbf{Set}^{\mathbf{C}%
}\right) ^{op}\mathbf{\times K}\longrightarrow \mathbf{K}^{\mathbf{C}^{op}}%
\mathbf{,}
\end{eqnarray*}%
while the last four assignments are covariant on both arguments:%
\begin{eqnarray*}
\_\otimes _{\mathbf{Set}}\_ &:&\mathbf{K}^{\mathbf{C}}\mathbf{\times Set}%
\longrightarrow \mathbf{K}^{\mathbf{C}}\mathbf{,} \\
\_\otimes _{\mathbf{Set}}\_ &:&\mathbf{Set\times K}^{\mathbf{C}%
}\longrightarrow \mathbf{K}^{\mathbf{C}}\mathbf{,} \\
\_\otimes _{\mathbf{Set}}\_ &:&\mathbf{Set}^{\mathbf{C}}\mathbf{\times K}%
\longrightarrow \mathbf{K}^{\mathbf{C}}\mathbf{,} \\
\_\otimes _{\mathbf{Set}}\_ &:&\mathbf{K\times Set}^{\mathbf{C}%
}\longrightarrow \mathbf{K}^{\mathbf{C}}\mathbf{.}
\end{eqnarray*}
\end{remark}

\begin{definition}
Let $\mathbf{K}$ be a complete and cocomplete category, let $\mathbf{C}$ be
a small category, and let%
\begin{equation*}
\mathcal{A}:\mathbf{C}\longrightarrow \mathbf{K,~}\mathcal{B}:\mathbf{C}%
\longrightarrow \mathbf{Set,~}\mathcal{F}:\mathbf{C}^{op}\longrightarrow 
\mathbf{Set,}
\end{equation*}%
be functors. Then define the following objects:

\begin{enumerate}
\item $Hom_{\mathbf{Set}^{\mathbf{C}}}\left( \mathcal{B},\mathcal{A}\right)
\in \mathbf{K}$ is the \textbf{end} \cite[Chapter IX.5]%
{Mac-Lane-Categories-1998-MR1712872}, of the bifunctor $\left( U,V\right)
\mapsto Hom_{\mathbf{Set}}\left( \mathcal{B}\left( U\right) ,\mathcal{A}%
\left( V\right) \right) $, i.e. 
\begin{equation*}
Hom_{\mathbf{Set}^{\mathbf{C}}}\left( \mathcal{B},\mathcal{A}\right) 
{:=}%
\ker \left( \dprod\limits_{U}Hom_{\mathbf{Set}}\left( \mathcal{B}\left(
U\right) ,\mathcal{A}\left( U\right) \right) \rightrightarrows
\dprod\limits_{U\rightarrow V}Hom_{\mathbf{Set}}\left( \mathcal{B}\left(
U\right) ,\mathcal{A}\left( V\right) \right) \right) .
\end{equation*}

\item $\mathcal{A\otimes }_{\mathbf{Set}^{\mathbf{C}}}\mathcal{F=F\otimes }_{%
\mathbf{Set}^{\mathbf{C}^{op}}}\mathcal{A}\in \mathbf{K}$ is the \textbf{%
coend }\cite[Chapter IX.6]{Mac-Lane-Categories-1998-MR1712872}, of the
bifunctor $\left( U,V\right) \mapsto \mathcal{A}\left( U\right) \otimes _{%
\mathbf{Set}}\mathcal{F}\left( V\right) $, i.e.%
\begin{equation*}
\mathcal{A\otimes }_{\mathbf{Set}^{\mathbf{C}}}\mathcal{F}%
{:=}%
coker\left( \dcoprod\limits_{U\rightarrow V}\mathcal{A}\left( U\right)
\otimes _{\mathbf{Set}}\mathcal{F}\left( V\right) \rightrightarrows
\dprod\limits_{U}\mathcal{A}\left( U\right) \otimes _{\mathbf{Set}}\mathcal{F%
}\left( U\right) \right) .
\end{equation*}
\end{enumerate}
\end{definition}

\begin{proposition}
\label{Prop-Hom-KC-(B-Tensor-Set-G,A)}Let $G\in \mathbf{K}$, $\mathcal{A}\in 
\mathbf{K}^{\mathbf{C}}$, and $\mathcal{B}\in \mathbf{Set}^{\mathbf{C}}$.
Then%
\begin{equation*}
Hom_{\mathbf{K}^{\mathbf{C}}}\left( \mathcal{B}\otimes _{\mathbf{Set}}G,%
\mathcal{A}\right) 
\simeq%
Hom_{\mathbf{K}}\left( G,Hom_{\mathbf{Set}^{\mathbf{C}}}\left( \mathcal{B},%
\mathcal{A}\right) \right) 
\simeq%
Hom_{\mathbf{Set}^{\mathbf{C}}}\left( \mathcal{B},Hom_{\mathbf{K}}\left( G,%
\mathcal{A}\right) \right)
\end{equation*}%
naturally on $G$, $\mathcal{A}$, and $\mathcal{B}$.
\end{proposition}

\begin{proof}
\begin{equation*}
Hom_{\mathbf{K}^{\mathbf{C}}}\left( \mathcal{B}\otimes _{\mathbf{Set}}G,%
\mathcal{A}\right) 
\simeq%
\end{equation*}%
\begin{equation*}
\simeq%
\ker \left( \dprod\limits_{U}Hom_{\mathbf{K}}\left( \left( \mathcal{B}%
\otimes _{\mathbf{Set}}G\right) \left( U\right) ,\mathcal{A}\left( U\right)
\right) \rightrightarrows \dprod\limits_{U\rightarrow V}Hom_{\mathbf{K}%
}\left( \left( \mathcal{B}\otimes _{\mathbf{Set}}G\right) \left( U\right) ,%
\mathcal{A}\left( V\right) \right) \right)
\end{equation*}%
\begin{equation*}
\simeq%
\ker \left( \dprod\limits_{U}\dprod\limits_{\mathcal{B}\left( U\right) }Hom_{%
\mathbf{K}}\left( G,\mathcal{A}\left( U\right) \right) \rightrightarrows
\dprod\limits_{U\rightarrow V}\dprod\limits_{\mathcal{B}\left( U\right)
}Hom_{\mathbf{K}}\left( G,\mathcal{A}\left( V\right) \right) \right)
\end{equation*}%
\begin{equation*}
\simeq%
Hom_{\mathbf{K}}\left( G,\ker \left( \dprod\limits_{U}\dprod\limits_{%
\mathcal{B}\left( U\right) }\mathcal{A}\left( U\right) \rightrightarrows
\dprod\limits_{U\rightarrow V}\dprod\limits_{\mathcal{B}\left( U\right) }%
\mathcal{A}\left( V\right) \right) \right)
\end{equation*}%
\begin{equation*}
\simeq%
Hom_{\mathbf{K}}\left( G,\ker \left( \dprod\limits_{U}Hom_{\mathbf{Set}%
}\left( \mathcal{B}\left( U\right) ,\mathcal{A}\left( U\right) \right)
\rightrightarrows \dprod\limits_{U\rightarrow V}Hom_{\mathbf{Set}}\left( 
\mathcal{B}\left( U\right) ,\mathcal{A}\left( V\right) \right) \right)
\right)
\end{equation*}%
\begin{equation*}
\simeq%
Hom_{\mathbf{K}}\left( G,Hom_{\mathbf{Set}^{\mathbf{C}}}\left( \mathcal{B},%
\mathcal{A}\right) \right) .
\end{equation*}%
Similarly,%
\begin{equation*}
Hom_{\mathbf{K}^{\mathbf{C}}}\left( \mathcal{B}\otimes _{\mathbf{Set}}G,%
\mathcal{A}\right) 
\simeq%
\end{equation*}%
\begin{equation*}
\simeq%
\ker \left( \dprod\limits_{U}\dprod\limits_{\mathcal{B}\left( U\right) }Hom_{%
\mathbf{K}}\left( G,\mathcal{A}\left( U\right) \right) \rightrightarrows
\dprod\limits_{U\rightarrow V}\dprod\limits_{\mathcal{B}\left( U\right)
}Hom_{\mathbf{K}}\left( G,\mathcal{A}\left( V\right) \right) \right)
\end{equation*}%
\begin{equation*}
\simeq%
\ker \left( \dprod\limits_{U}Hom_{\mathbf{Set}}\left( \mathcal{B}\left(
U\right) ,Hom_{\mathbf{K}}\left( G,\mathcal{A}\left( U\right) \right)
\right) \rightrightarrows \dprod\limits_{U\rightarrow V}Hom_{\mathbf{Set}%
}\left( \mathcal{B}\left( U\right) ,Hom_{\mathbf{K}}\left( G,\mathcal{A}%
\left( V\right) \right) \right) \right)
\end{equation*}%
\begin{equation*}
\simeq%
Hom_{\mathbf{Set}^{\mathbf{C}}}\left( \mathcal{B},Hom_{\mathbf{K}}\left( G,%
\mathcal{A}\right) \right) .
\end{equation*}
\end{proof}

Proposition below is a variant of Yoneda's Lemma:

\begin{proposition}
\label{Prop-Variant-of-Yoneda}Let $G\in \mathbf{K}$, $U\in \mathbf{C},$ and $%
\mathcal{A}\in \mathbf{K}^{\mathbf{C}^{op}}$. Then%
\begin{equation*}
Hom_{\mathbf{K}^{\mathbf{C}^{op}}}\left( h_{U}\otimes _{\mathbf{Set}}G,%
\mathcal{A}\right) 
\simeq%
Hom_{\mathbf{K}}\left( G,\mathcal{A}\left( U\right) \right) =\left( Hom_{%
\mathbf{K}}\left( G,\mathcal{A}\right) \right) \left( U\right)
\end{equation*}%
naturally on $G$, $U$, and $\mathcal{A}$.
\end{proposition}

\begin{proof}
Using the Yoneda isomorphism $Hom_{\mathbf{Set}^{\mathbf{C}^{op}}}\left(
h_{U},\mathcal{A}\right) 
\simeq%
\mathcal{A}\left( U\right) $, and Proposition \ref%
{Prop-Hom-KC-(B-Tensor-Set-G,A)}, one gets%
\begin{equation*}
Hom_{\mathbf{K}^{\mathbf{C}^{op}}}\left( h_{U}\otimes _{\mathbf{Set}}G,%
\mathcal{A}\right) 
\simeq%
Hom_{\mathbf{K}}\left( G,Hom_{\mathbf{Set}^{\mathbf{C}^{op}}}\left( h_{U},%
\mathcal{A}\right) \right) 
\simeq%
Hom_{\mathbf{K}}\left( G,\mathcal{A}\left( U\right) \right) .
\end{equation*}
\end{proof}

\section{Grothendieck topologies and (pre)sheaves}

\begin{definition}
A \textbf{sieve} $R$ over $U\in \mathbf{C}$ is a subfunctor $R\subseteq
h_{U} $ of 
\begin{equation*}
h_{U}=Hom_{\mathbf{C}}\left( \_,U\right) :\mathbf{C}^{op}\longrightarrow 
\mathbf{Set.}
\end{equation*}
\end{definition}

\begin{remark}
Compare with \cite[Definition 16.1.1]{Kashiwara-Categories-MR2182076}.
\end{remark}

\begin{definition}
\label{Def-Site}A Grothendieck site (or simply a \textbf{site}) $X$ is a
pair $\left( \mathbf{C}_{X},Cov\left( X\right) \right) $ where $\mathbf{C}%
_{X}$ is a category, and%
\begin{equation*}
Cov\left( X\right) =\dbigcup\limits_{U\in \mathbf{C}_{X}}Cov\left( U\right) ,
\end{equation*}%
where $Cov\left( U\right) $ are the sets of \textbf{covering sieves} over $U$%
, satisfying the axioms GT1-GT4 from \cite[Definition 16.1.2]%
{Kashiwara-Categories-MR2182076}, or, equivalently, the axioms T1-T3 from 
\cite[Definition II.1.1]{SGA4-1-MR0354652}. The site is called \textbf{small}
iff $\mathbf{C}_{X}$ is a small category.
\end{definition}

\begin{remark}
The class (or a set, if $X$ is small) $Cov\left( X\right) $ is called the 
\textbf{topology} on $X$.
\end{remark}

\begin{notation}
Given $U\in \mathbf{C}_{X}$, and $R\in Cov\left( X\right) $, denote simply%
\begin{equation*}
\mathbf{C}_{U}%
{:=}%
\left( \mathbf{C}_{X}\right) _{U},~\mathbf{C}_{R}%
{:=}%
\left( \mathbf{C}_{X}\right) _{R}
\end{equation*}%
(see Definitions \ref{Def-Comma-U} and \ref{Def-Comma-R}).
\end{notation}

\begin{proposition}
Let $G\in \mathbf{K}$, and let $R\subseteq h_{U}$ be a sieve. Then

\begin{enumerate}
\item 
\begin{equation*}
Hom_{\mathbf{K}^{\mathbf{C}^{op}}}\left( G\otimes _{\mathbf{Set}}R,\mathcal{A%
}\right) 
\simeq%
Hom_{\mathbf{K}}\left( G,Hom_{\mathbf{Set}^{\mathbf{C}^{op}}}\left( R,%
\mathcal{A}\right) \right) 
\simeq%
Hom_{\mathbf{K}}\left( G,\underset{\left( V\rightarrow U\right) \in \mathbf{C%
}_{R}}{\underleftarrow{\lim }}\mathcal{A}\left( V\right) \right) .
\end{equation*}

\item 
\begin{equation*}
Hom_{\mathbf{Set}^{\mathbf{C}^{op}}}\left( R,\mathcal{A}\right) 
\simeq%
\underset{\left( V\rightarrow U\right) \in \mathbf{C}_{R}}{\underleftarrow{%
\lim }}\mathcal{A}\left( V\right) .
\end{equation*}
\end{enumerate}
\end{proposition}

\begin{proof}
~

\begin{enumerate}
\item It follows from Proposition \ref{Prop-Hom-KC-(B-Tensor-Set-G,A)} and 
\cite[Corollary I.3.5]{SGA4-1-MR0354652}, that, naturally on $G\in \mathbf{K}
$,%
\begin{eqnarray*}
&&Hom_{\mathbf{K}^{\mathbf{C}^{op}}}\left( G\otimes _{\mathbf{Set}}R,%
\mathcal{A}\right) 
\simeq%
Hom_{\mathbf{K}}\left( G,Hom_{\mathbf{Set}^{\mathbf{C}^{op}}}\left( R,%
\mathcal{A}\right) \right) \\
&&%
\simeq%
Hom_{\mathbf{Set}^{\mathbf{C}^{op}}}\left( R,Hom_{\mathbf{K}}\left( G,%
\mathcal{A}\right) \right) 
\simeq%
\underset{\left( V\rightarrow U\right) \in \mathbf{C}_{R}}{\underleftarrow{%
\lim }}Hom_{\mathbf{K}}\left( G,\mathcal{A}\right) \left( V\right) \\
&&%
\simeq%
\underset{\left( V\rightarrow U\right) \in \mathbf{C}_{R}}{\underleftarrow{%
\lim }}Hom_{\mathbf{K}}\left( G,\mathcal{A}\left( V\right) \right) 
\simeq%
Hom_{\mathbf{K}}\left( G,\underset{\left( V\rightarrow U\right) \in \mathbf{C%
}_{R}}{\underleftarrow{\lim }}\mathcal{A}\left( V\right) \right) .
\end{eqnarray*}

\item Let%
\begin{equation*}
K=Hom_{\mathbf{Set}^{\mathbf{C}^{op}}}\left( R,\mathcal{A}\right) \in 
\mathbf{K},~L=\underset{\left( V\rightarrow U\right) \in \mathbf{C}_{R}}{%
\underleftarrow{\lim }}\mathcal{A}\left( V\right) \in \mathbf{K}.
\end{equation*}%
We have just proved that $h_{K}%
\simeq%
h_{L}\in \mathbf{Set}^{\mathbf{K}^{op}}$. It follows from Yoneda's lemma
that $K%
\simeq%
L$.
\end{enumerate}
\end{proof}

\begin{definition}
\label{Def-Pretopology}We say that the topology on a small site $X$ is
induced by a \textbf{pretopology} if each object $U\in \mathbf{C}_{X}$ is
supplied with \textbf{covers} $\left\{ U_{i}\rightarrow U\right\} _{i\in I}$%
, satisfying \cite[Definition II.1.3]{SGA4-1-MR0354652} (compare to \cite[%
Definition 16.1.5]{Kashiwara-Categories-MR2182076}), and the covering sieves 
$R\in Cov\left( X\right) $ are \textbf{generated} by covers:%
\begin{equation*}
R=R_{\left\{ U_{i}\rightarrow U\right\} }\subseteq h_{U},
\end{equation*}%
where $R_{\left\{ U_{i}\rightarrow U\right\} }\left( V\right) $ consists of
morphisms $\left( V\rightarrow U\right) \in h_{U}\left( V\right) $ admitting
a decomposition%
\begin{equation*}
\left( V\rightarrow U\right) =\left( V\rightarrow U_{i}\rightarrow U\right) .
\end{equation*}
\end{definition}

\begin{remark}
We use the word \textbf{covers} for general sites, and reserve the word 
\textbf{coverings} for open coverings of topological spaces.
\end{remark}

\begin{example}
\label{Site-TOP}Let $X$ be a topological space. We will call the site $%
OPEN\left( X\right) $ below the \textbf{standard site} for $X$:%
\begin{equation*}
OPEN\left( X\right) =\left( \mathbf{C}_{OPEN\left( X\right) },Cov\left(
OPEN\left( X\right) \right) \right) .
\end{equation*}%
$\mathbf{C}_{OPEN\left( X\right) }$ has open subsets of $X$ as objects and
inclusions $U\subseteq V$ as morphisms. The pretopology on $OPEN\left(
X\right) $ consists of families%
\begin{equation*}
\left\{ U_{i}\subseteq U\right\} _{i\in I}\in \mathbf{C}_{OPEN\left(
X\right) }
\end{equation*}%
with%
\begin{equation*}
\dbigcup\limits_{i\in I}U_{i}=U.
\end{equation*}%
The corresponding topology consists of sieves $R_{\left\{ U_{i}\subseteq
U\right\} }\subseteq h_{U}$ where%
\begin{equation*}
\left( V\subseteq U\right) \in R_{\left\{ U_{i}\subseteq U\right\} }\left(
U\right) \iff \exists i\in I~\left( V\subseteq U_{i}\right) .
\end{equation*}
\end{example}

\begin{remark}
\label{Denote-standard-site-simply}We will often denote the standard site
simply by $X=\left( \mathbf{C}_{X},Cov\left( X\right) \right) $.
\end{remark}

\begin{example}
\label{Site-NORM}Let again $X$ be a topological space. Consider the site%
\begin{equation*}
NORM\left( X\right) =\left( \mathbf{C}_{NORM\left( X\right) },Cov\left(
NORM\left( X\right) \right) \right)
\end{equation*}%
where $\mathbf{C}_{NORM\left( X\right) }=\mathbf{C}_{X}$, while the
pretopology on $NORM\left( X\right) $ consists of \textbf{normal}
(Definition \ref{Def-Normal-covering}) coverings $\left\{ U_{i}\subseteq
U\right\} $.
\end{example}

\begin{example}
\label{Site-FINITE}Let again $X$ be a topological space. Consider the site%
\begin{equation*}
FINITE\left( X\right) =\left( \mathbf{C}_{FINITE\left( X\right) },Cov\left(
FINITE\left( X\right) \right) \right)
\end{equation*}%
where $\mathbf{C}_{FINITE\left( X\right) }=\mathbf{C}_{X}$, while the
pretopology on $FINITE\left( X\right) $ consists of \textbf{finite} normal
coverings $\left\{ U_{i}\subseteq U\right\} $.
\end{example}

\begin{example}
\label{Equivariant}Let $G$ be a topological group, and $X$ be a $G$-space.
The corresponding site $OPEN_{G}\left( X\right) $ has $G$-invariant open
subsets of $X$ as objects of $\mathbf{C}_{OPEN_{G}\left( X\right) }$ and the
pretopology consisting of $G$-invariant open coverings (compare to \cite[%
Example 1.1.4]{Artin-GT}, or \cite[Example (1.3.2)]{Tamme-MR1317816}).
\end{example}

\begin{example}
\label{Site-ETALE}Let $X$ be a noetherian scheme, and define the site $%
X^{et} $ by: $\mathbf{C}_{X^{et}}$ is the category of schemes $Y/X$ \'{e}%
tale, finite type, while the pretopology on $X^{et}$ consists of finite
surjective families of maps. See \cite[Example 1.1.6]{Artin-GT}, or \cite[%
II.1.2]{Tamme-MR1317816}.
\end{example}

\begin{definition}
\label{Def-(pre)sheaves}Let $X=\left( \mathbf{C}_{X},Cov\left( X\right)
\right) $ be a small site, and let $\mathbf{K}$ be a complete category.

\begin{enumerate}
\item A \textbf{presheaf} $\mathcal{A}$ on $X$ with values in $\mathbf{K}$
is a functor $\mathcal{A}:\left( \mathbf{C}_{X}\right) ^{op}\rightarrow 
\mathbf{K.}$

\item A presheaf $\mathcal{A}$ is \textbf{separated} provided that for any $%
U\in \mathbf{C}_{X}$%
\begin{equation*}
\mathcal{A}\left( U\right) =Hom_{\mathbf{Set}^{\left( \mathbf{C}_{X}\right)
^{op}}}\left( h_{U},\mathcal{A}\right) \longrightarrow Hom_{\mathbf{Set}%
^{\left( \mathbf{C}_{X}\right) ^{op}}}\left( R,\mathcal{A}\right)
\end{equation*}%
is a monomorphism for any covering sieve $R$ over $U$.

\item A presheaf $\mathcal{A}$ is a \textbf{sheaf} provided that for any $%
U\in \mathbf{C}_{X}$%
\begin{equation*}
\mathcal{A}\left( U\right) =Hom_{\mathbf{Set}^{\left( \mathbf{C}_{X}\right)
^{op}}}\left( h_{U},\mathcal{A}\right) \longrightarrow Hom_{\mathbf{Set}%
^{\left( \mathbf{C}_{X}\right) ^{op}}}\left( R,\mathcal{A}\right)
\end{equation*}%
is an isomorphism for any covering sieve $R$ over $U$.
\end{enumerate}
\end{definition}

\begin{notation}
\label{Not-(Pre)sheaves}Let $\mathbf{pS}\left( X,\mathbf{K}\right) =\mathbf{K%
}^{\left( \mathbf{C}_{X}\right) ^{op}}$ be the category of presheaves on $X$
with values in $\mathbf{K}$, and let $\mathbf{S}\left( X,\mathbf{K}\right) $
be the full subcategory of sheaves.
\end{notation}

\begin{proposition}
\label{Prop-Pretopology}Let $X$ be a small site with a pretopology
(Definition \ref{Def-Pretopology}). Then for any presheaf $\mathcal{A}$ with
values in a complete category $\mathbf{K}$, and for any sieve $R$ generated
by a cover $\left\{ U_{i}\rightarrow U\right\} _{i\in I}$,%
\begin{equation*}
Hom_{\mathbf{Set}^{\mathbf{C}^{op}}}\left( R,\mathcal{A}\right) 
\simeq%
\underset{\left( V\rightarrow U\right) \in \mathbf{C}_{R}}{\underleftarrow{%
\lim }}\mathcal{A}\left( V\right) 
\simeq%
\ker \left( \dprod\limits_{i\in I}\mathcal{A}\left( U_{i}\right)
\rightrightarrows \dprod\limits_{i,j\in I}\mathcal{A}\left( U_{i}\times
_{U}U_{j}\right) \right) .
\end{equation*}
\end{proposition}

\begin{proof}
Let $G\in \mathbf{K}$, and $\mathcal{B}=Hom_{\mathbf{K}}\left( G,\mathcal{A}%
\right) $. $\mathcal{B}$ is a presheaf of sets on $X$. It follows from \cite[%
Proposition I.2.12]{SGA4-1-MR0354652}, that 
\begin{eqnarray*}
&&Hom_{\mathbf{C}}\left( G,Hom_{\mathbf{Set}^{\mathbf{C}^{op}}}\left( R,%
\mathcal{A}\right) \right) 
\simeq%
Hom_{\mathbf{Set}^{\mathbf{C}^{op}}}\left( R,\mathcal{B}\right) 
\simeq%
\underset{\left( V\rightarrow U\right) \in \mathbf{C}_{R}}{\underleftarrow{%
\lim }}\mathcal{B}\left( V\right) \\
&&%
\simeq%
Hom_{\mathbf{C}}\left( G,\underset{\left( V\rightarrow U\right) \in \mathbf{C%
}_{R}}{\underleftarrow{\lim }}\mathcal{A}\left( V\right) \right) 
\simeq%
\ker \left( \dprod\limits_{i\in I}\mathcal{B}\left( U_{i}\right)
\rightrightarrows \dprod\limits_{i,j\in I}\mathcal{B}\left( U_{i}\times
_{U}U_{j}\right) \right) \\
&&%
\simeq%
Hom_{\mathbf{K}}\left( G,\ker \left( \dprod\limits_{i\in I}\mathcal{A}\left(
U_{i}\right) \rightrightarrows \dprod\limits_{i,j\in I}\mathcal{A}\left(
U_{i}\times _{U}U_{j}\right) \right) \right) ,
\end{eqnarray*}%
naturally on $G$. Apply Yoneda's lemma.
\end{proof}

\begin{definition}
\label{Def-Plus-sheaf}Let $X$ be a small site, and $\mathbf{K}$ be a locally
finitely presentable category. For a presheaf $\mathcal{A}$, define a
presheaf $\left( \mathcal{A}\right) _{\mathbf{K}}^{+}$ (or simply $\mathcal{A%
}^{+})$ by the following:%
\begin{equation*}
\mathcal{A}^{+}\left( U\right) =\underset{R}{\underrightarrow{\lim }}~%
\underset{\left( V\rightarrow U\right) \in \mathbf{C}_{R}}{\underleftarrow{%
\lim }}~\mathcal{A}\left( V\right) =\underset{R}{\underrightarrow{\lim }}%
~Hom_{\mathbf{Set}^{\left( \mathbf{C}_{X}\right) ^{op}}}\left( R,\mathcal{A}%
\right)
\end{equation*}%
where $R\subseteq h^{U}$ runs over all covering sieves over $U$.
\end{definition}

\begin{proposition}
If the topology is induced by a pretopology (Definition \ref{Def-Pretopology}%
), then%
\begin{equation*}
\mathcal{A}^{+}\left( U\right) =\underset{\left\{ U_{i}\rightarrow U\right\} 
}{\underrightarrow{\lim }}~\ker \left( \dprod\limits_{i}\mathcal{A}\left(
U_{i}\right) \rightrightarrows \dprod\limits_{i,j}\mathcal{A}\left(
U_{i}\times _{U}U_{j}\right) \right)
\end{equation*}%
where $\left\{ U_{i}\rightarrow U\right\} $ runs over the covers of $U$.
\end{proposition}

\begin{proof}
Follows from Proposition \ref{Prop-Pretopology}.
\end{proof}

\begin{proposition}
\label{Prop-Hom(G,_)}~

\begin{enumerate}
\item If $G\in \mathbf{K}$ is a finitely presentable object, then there is a
natural isomorphism%
\begin{equation*}
Hom_{\mathbf{K}}\left( G,\mathcal{A}^{+}\right) 
\simeq%
Hom_{\mathbf{K}}\left( G,\mathcal{A}\right) ^{+}.
\end{equation*}

\item If $\mathfrak{G}\subseteq \mathbf{K}$ is a \textbf{strong generator} 
\cite[Definition 0.6]%
{Adamek-Rosicky-1994-Locally-presentable-categories-MR1294136}, then:

\begin{enumerate}
\item $\mathcal{A}\in \mathbf{pS}\left( X,\mathbf{K}\right) $ is separated
iff $Hom_{\mathbf{K}}\left( G,\mathcal{A}\right) \in \mathbf{pS}\left( X,%
\mathbf{Set}\right) $ is separated for any $G\in \mathfrak{G}$.

\item $\mathcal{A}\in \mathbf{pS}\left( X,\mathbf{K}\right) $ is a sheaf iff 
$Hom_{\mathbf{K}}\left( G,\mathcal{A}\right) \in \mathbf{pS}\left( X,\mathbf{%
Set}\right) $ is a sheaf for any $G\in \mathfrak{G}$.
\end{enumerate}
\end{enumerate}
\end{proposition}

\begin{proof}
~

\begin{enumerate}
\item If $G\in Pres_{\aleph _{0}}\mathbf{K}$ (see Remark \ref%
{Rem-Locally-presentable}), then the functor $Hom_{\mathbf{K}}\left(
G,\_\right) $ commutes with directed colimits and arbitrary limits.
Therefore, $Hom_{\mathbf{K}}\left( G,\mathcal{A}\right) ^{+}%
\simeq%
Hom_{\mathbf{K}}\left( G,\mathcal{A}^{+}\right) $.

\item $Hom_{\mathbf{K}}\left( G,\_\right) $ commutes with arbitrary limits.
Therefore, for any covering sieve $R\subseteq h_{U}$,%
\begin{equation*}
Hom_{\mathbf{K}}\left( G,\underset{\left( V\rightarrow U\right) \in \mathbf{C%
}_{R}}{\underleftarrow{\lim }}\mathcal{A}\left( V\right) \right) 
\simeq%
\underset{\left( V\rightarrow U\right) \in \mathbf{C}_{R}}{\underleftarrow{%
\lim }}Hom_{\mathbf{K}}\left( G,\mathcal{A}\left( V\right) \right) .
\end{equation*}%
The morphism%
\begin{equation*}
\mathcal{A}\left( U\right) \longrightarrow \underset{\left( V\rightarrow
U\right) \in \mathbf{C}_{R}}{\underleftarrow{\lim }}\mathcal{A}\left(
V\right)
\end{equation*}%
is a monomorphism (respectively, an isomorphism) iff%
\begin{equation*}
Hom_{\mathbf{K}}\left( G,\mathcal{A}\left( U\right) \right) \longrightarrow 
\underset{\left( V\rightarrow U\right) \in \mathbf{C}_{R}}{\underleftarrow{%
\lim }}Hom_{\mathbf{K}}\left( G,\mathcal{A}\left( V\right) \right)
\end{equation*}%
is a monomorphism (respectively, an isomorphism) for any $G\in \mathfrak{G}$.
\end{enumerate}
\end{proof}

\begin{theorem}
\label{Th-Properties-of-Plus}Assume that $\mathbf{K}$ is a finitely
presentable category (Definition \ref{Def-Locally-presentable}). Let $%
\lambda \left( \mathcal{A}\right) :\mathcal{A\rightarrow A}^{+}$ be the
canonical morphism of functors%
\begin{equation*}
\mathbf{1}_{\mathbf{pS}\left( X,\mathbf{K}\right) }\longrightarrow \left(
{}\right) ^{+}:\mathbf{pS}\left( X,\mathbf{K}\right) \longrightarrow \mathbf{%
pS}\left( X,\mathbf{K}\right) .
\end{equation*}%
Then:

\begin{enumerate}
\item The functor $\left( {}\right) ^{+}$ is left exact.

\item For any $\mathcal{A}$, $\mathcal{A}^{+}$ is a separated presheaf.

\item A presheaf $\mathcal{A}$ is separated iff $\lambda \left( \mathcal{A}%
\right) $ is a monomorphism. In that case $\mathcal{A}^{+}$ is a sheaf.

\item The following conditions are equivalent:

\begin{enumerate}
\item $\lambda \left( \mathcal{A}\right) $ is an isomorphism.

\item $\mathcal{A}$ is a sheaf.
\end{enumerate}

\item The functor $\left( {}\right) _{\mathbf{K}}^{\#}=\left( {}\right) _{%
\mathbf{K}}^{++}$ is left adjoint to the inclusion%
\begin{equation*}
i_{X,\mathbf{K}}:\mathbf{S}\left( X,\mathbf{K}\right) \hookrightarrow 
\mathbf{pS}\left( X,\mathbf{K}\right) .
\end{equation*}
\end{enumerate}
\end{theorem}

\begin{proof}
Let $Pres_{\aleph _{0}}\mathbf{K}$ be a set of representatives for the
isomorphism classes of finitely presentable objects of $\mathbf{K}$ (see
Remark \ref{Rem-Locally-presentable}). This set forms a \textbf{strong
generator} \cite[Theorem 1.20]%
{Adamek-Rosicky-1994-Locally-presentable-categories-MR1294136}, for $\mathbf{%
K}$.

\begin{enumerate}
\item The functor $\mathcal{A}\mapsto \mathcal{A}^{+}$ is the composition of
a limit $\underset{\left( V\rightarrow U\right) \in \mathbf{C}_{R}}{%
\underleftarrow{\lim }}\mathcal{A}\left( V\right) $ which commutes with
arbitrary limits, and a directed colimit $\underset{R\subseteq h_{U}}{%
\underrightarrow{\lim }}$ which commutes with \textbf{finite} limits \cite[%
Proposition 1.59]%
{Adamek-Rosicky-1994-Locally-presentable-categories-MR1294136}. Therefore, $%
\left( {}\right) ^{+}$ is left exact (commutes with finite limits).

\item Due to \cite[Proposition II.3.2]{SGA4-1-MR0354652}, $Hom_{\mathbf{K}%
}\left( G,\mathcal{A}^{+}\right) $ is separated for any $G\in Pres_{\aleph
_{0}}\mathbf{K}$. Apply Proposition \ref{Prop-Hom(G,_)}.

\item Due to \cite[Proposition II.3.2]{SGA4-1-MR0354652}, 
\begin{equation*}
Hom_{\mathbf{K}}\left( G,\mathcal{A}\right) \longrightarrow Hom_{\mathbf{K}%
}\left( G,\mathcal{A}^{+}\right)
\end{equation*}%
is a monomorphism iff $Hom_{\mathbf{K}}\left( G,\mathcal{A}\right) $ is
separated for any $G\in Pres_{\aleph _{0}}\mathbf{K}$. In that case $Hom_{%
\mathbf{K}}\left( G,\mathcal{A}^{+}\right) $ is a sheaf. Apply Proposition %
\ref{Prop-Hom(G,_)}.

\item Due to \cite[Proposition II.3.2]{SGA4-1-MR0354652}, 
\begin{equation*}
Hom_{\mathbf{K}}\left( G,\mathcal{A}\right) \longrightarrow Hom_{\mathbf{K}%
}\left( G,\mathcal{A}^{+}\right)
\end{equation*}%
is an isomorphism iff $Hom_{\mathbf{K}}\left( G,\mathcal{A}\right) $ is a
sheaf for any $G\in Pres_{\aleph _{0}}\mathbf{K}$. Apply Proposition \ref%
{Prop-Hom(G,_)}.

\item We need to prove that for any sheaf $\mathcal{B}$, any morphism $%
\mathcal{A\rightarrow B}$ has a unique decomposition%
\begin{equation*}
\mathcal{A\longrightarrow A}^{++}\longrightarrow \mathcal{B}.
\end{equation*}%
The existence is easy: since $\mathcal{B\longrightarrow B}^{++}$ is an
isomorphism, take the decomposition%
\begin{equation*}
\mathcal{A\longrightarrow A}^{++}\longrightarrow \mathcal{B}^{++}%
\simeq%
\mathcal{B}.
\end{equation*}%
To prove uniqueness, consider two decompositions%
\begin{equation*}
\begin{diagram} \mathcal{A} & \rTo & \mathcal{A}^{++} & \pile{\rTo^{\alpha}
\\ \rTo_{\beta}} & \mathcal{B} \\ \end{diagram}
\end{equation*}%
and apply $Hom_{\mathbf{K}}\left( G,\_\right) $:%
\begin{equation*}
\begin{diagram} Hom_{\QTR{bf}{K}}\left( G,\mathcal{A}\right) & \rTo &
Hom_{\QTR{bf}{K}}\left( G,\mathcal{A}\right) ^{++} &
\pile{\rTo^{Hom_{\QTR{bf}{K}}\left( G,\alpha \right) } \\
\rTo_{Hom_{\QTR{bf}{K}}\left( G,\beta \right)}} & Hom_{\QTR{bf}{K}}\left(
G,\mathcal{B}\right). \\ \end{diagram}
\end{equation*}%
It follows that $Hom_{\mathbf{K}}\left( G,\alpha \right) =Hom_{\mathbf{K}%
}\left( G,\beta \right) $ for any $G\in Pres_{\aleph _{0}}\mathbf{K}$,
therefore $\alpha =\beta $.
\end{enumerate}
\end{proof}

\begin{theorem}
\label{Th-Sheaves-K-locally-presentable}Let $X$ be a small site, and $%
\mathbf{K}$ be a locally $\lambda $-presentable category. Then%
\begin{equation*}
\mathbf{S}\left( X,\mathbf{K}\right) \subseteq \mathbf{pS}\left( X,\mathbf{K}%
\right)
\end{equation*}%
is a reflective subcategory.
\end{theorem}

\begin{proof}
Due to \cite[Corollary 1.54]%
{Adamek-Rosicky-1994-Locally-presentable-categories-MR1294136}, $\mathbf{pS}%
\left( X,\mathbf{K}\right) =\mathbf{K}^{\mathbf{C}_{X}}$ is a locally $%
\lambda $-presentable category. For each covering sieve $R\subseteq h_{U}$
and each $G\in \mathbf{K}$, let%
\begin{equation*}
g_{R,G}:G\otimes _{\mathbf{Set}}R\longrightarrow G\otimes _{\mathbf{Set}%
}h_{U}
\end{equation*}%
be the corresponding morphism in $\mathbf{pS}\left( X,\mathbf{K}\right) $.
For a presheaf $\mathcal{A}$, apply $Hom_{\mathbf{pS}\left( X,\mathbf{K}%
\right) }\left( \_,\mathcal{A}\right) $:%
\begin{eqnarray*}
&&Hom_{\mathbf{pS}\left( X,\mathbf{K}\right) }\left( G\otimes _{\mathbf{Set}%
}R,\mathcal{A}\right) 
\simeq%
Hom_{\mathbf{pS}\left( X,\mathbf{Set}\right) }\left( R,Hom_{\mathbf{Set}%
}\left( G,\mathcal{A}\right) \right) , \\
&&Hom_{\mathbf{pS}\left( X,\mathbf{K}\right) }\left( G\otimes _{\mathbf{Set}%
}h_{U},\mathcal{A}\right) 
\simeq%
Hom_{\mathbf{K}}\left( G,\mathcal{A}\left( U\right) \right) =Hom_{\mathbf{K}%
}\left( G,\mathcal{A}\right) \left( U\right) , \\
&&Hom_{\mathbf{pS}\left( X,\mathbf{K}\right) }\left( g_{R,G},\mathcal{A}%
\right) 
\simeq%
\left( Hom_{\mathbf{K}}\left( G,\mathcal{A}\right) \left( U\right)
\longrightarrow Hom_{\mathbf{pS}\left( X,\mathbf{Set}\right) }\left( R,Hom_{%
\mathbf{K}}\left( G,\mathcal{A}\right) \right) \right) .
\end{eqnarray*}%
Assume that $G$ runs over $Pres_{\lambda }\mathbf{K}$. Then the following
conditions are equivalent:

\begin{enumerate}
\item $Hom_{\mathbf{pS}\left( X,\mathbf{K}\right) }\left( g_{R,G},\mathcal{A}%
\right) $ is a bijection for all $g_{R,G}$.

\item $Hom_{\mathbf{K}}\left( G,\mathcal{A}\right) $ is a sheaf of sets for
all $G\in Pres_{\lambda }\left( \mathbf{K}\right) $.

\item $\mathcal{A}$ is a sheaf.
\end{enumerate}

Choose a regular cardinal $\mu \geq \lambda $ such that for all $G$ both $%
G\otimes _{\mathbf{Set}}R$ and $G\otimes _{\mathbf{Set}}h_{U}$ are $\mu $%
-presentable. It follows that $\mathbf{S}\left( X,\mathbf{K}\right)
\subseteq \mathbf{pS}\left( X,\mathbf{K}\right) $ is the $\mu $%
-orthogonality class $\left\{ g_{R,G}\right\} ^{\perp }$ \cite[Definition
1.35]{Adamek-Rosicky-1994-Locally-presentable-categories-MR1294136}, in $%
\mathbf{pS}\left( X,\mathbf{K}\right) $, and therefore \cite[Theorem 1.39]%
{Adamek-Rosicky-1994-Locally-presentable-categories-MR1294136}, $\mathbf{S}%
\left( X,\mathbf{K}\right) $ is a reflective subcategory of $\mathbf{pS}%
\left( X,\mathbf{K}\right) $.
\end{proof}

\subsection{(Pre)sheaves on topological spaces}

Throughout this Subsection, $X$ is a topological space considered as the
site $OPEN\left( X\right) $ (see Example \ref{Site-TOP} and Remark \ref%
{Denote-standard-site-simply}).

\begin{definition}
\label{Def-Stalk}Assume that a category $\mathbf{K}$ admits filtered
colimits. Let $\mathcal{A}$ be a presheaf with values in $\mathbf{K}$, and
let $x\in X$. The \textbf{stalk} of $\mathcal{A}$ at $x$ is 
\begin{equation*}
\mathcal{A}_{x}%
{:=}%
\underrightarrow{\lim }_{U\in J\left( x\right) }\mathcal{A}\left( U\right)
\end{equation*}%
where $J\left( x\right) $ is the family of open neighborhoods of $x$.
\end{definition}

\begin{remark}
\label{Rem-Stalk-different-categories}In a situation when $K\subseteq L$ is
a subcategory, and $\mathcal{A}\in \mathbf{pS}\left( X,\mathbf{K}\right) $,
we will use notations $\left( \mathcal{A}\right) _{x}^{\mathbf{K}}$ and $%
\left( \mathcal{A}\right) _{x}^{\mathbf{L}}$ depending on whether the
colimit is taken in the category $\mathbf{K}$ or in the category $\mathbf{L}$%
.
\end{remark}

\begin{definition}
\label{Def-Local-isomorphism-presheaves}Let $\mathbf{K}$ admit filtered
colimits, and let $f:\mathcal{A\rightarrow B}$ be a morphism in the category
of presheaves $\mathbf{pS}\left( X,\mathbf{K}\right) $. We say that $f$ is a 
\textbf{local isomorphism} iff $f_{x}:\mathcal{A}_{x}\rightarrow \mathcal{B}%
_{x}$ is an isomorphism for any $x\in X$. In a situation when $\mathbf{K}%
\subseteq \mathbf{L}$, and%
\begin{equation*}
\mathcal{A},\mathcal{B}\in \mathbf{pS}\left( X,\mathbf{K}\right) \subseteq 
\mathbf{pS}\left( X,\mathbf{L}\right) ,
\end{equation*}%
we will say that $f$ is $\mathbf{K}$-local (respectively $\mathbf{L}$-local)
isomorphism iff $\left( f\right) _{x}^{\mathbf{K}}:\left( \mathcal{A}\right)
_{x}^{\mathbf{K}}\longrightarrow \left( \mathcal{B}\right) _{x}^{\mathbf{K}}$
(respectively $\left( f\right) _{x}^{\mathbf{L}}:\left( \mathcal{A}\right)
_{x}^{\mathbf{L}}\longrightarrow \left( \mathcal{B}\right) _{x}^{\mathbf{L}}$%
) is an isomorphism for any $x\in X$.
\end{definition}

\begin{proposition}
\label{Prop-Sheafification-is-local-iso}Let $\mathbf{K}$ be a complete
category admitting filtered colimits. Assume that $\mathbf{S}\left( X,%
\mathbf{K}\right) \subseteq \mathbf{pS}\left( X,\mathbf{K}\right) $ is
reflective, and the reflection is given by the functor%
\begin{equation*}
\left( {}\right) ^{\#}:\mathbf{pS}\left( X,\mathbf{K}\right) \longrightarrow 
\mathbf{S}\left( X,\mathbf{K}\right) .
\end{equation*}%
Then for any presheaf $\mathcal{A}$, the natural morphism $\mathcal{A}%
\rightarrow \mathcal{A}^{\#}$ is a local isomorphism.
\end{proposition}

\begin{proof}
Let $x\in X$, and $G\in \mathbf{K}$. Denote by $\mathcal{P}_{x,G}$ the
following \textbf{pointed} presheaf: $\mathcal{P}_{x,G}\left( U\right) $ is
a terminal object $T$ when $x\not\in U$, and $\mathcal{P}_{x,G}\left(
U\right) =G$ when $x\in U$. It is easy to check that $\mathcal{P}_{x,G}$ is
in fact a \textbf{sheaf}, and that for any presheaf $\mathcal{C}$,%
\begin{equation*}
Hom_{\mathbf{pS}\left( X,\mathbf{K}\right) }\left( \mathcal{C},\mathcal{P}%
_{x,G}\right) 
\simeq%
\underleftarrow{\lim }_{U\in J\left( x\right) }Hom_{\mathbf{K}}\left( 
\mathcal{C}\left( U\right) ,G\right) 
\simeq%
Hom_{\mathbf{K}}\left( \mathcal{C}_{x},G\right) ,
\end{equation*}%
naturally on $G$ and $\mathcal{C}$. Using the adjointness isomorphism, one
gets%
\begin{equation*}
Hom_{\mathbf{K}}\left( \mathcal{A}_{x},G\right) 
\simeq%
Hom_{\mathbf{pS}\left( X,\mathbf{K}\right) }\left( \mathcal{A},\mathcal{P}%
_{x,G}\right) 
\simeq%
Hom_{\mathbf{S}\left( X,\mathbf{K}\right) }\left( \mathcal{A}^{\#},\mathcal{P%
}_{x,G}\right) 
\simeq%
Hom_{\mathbf{K}}\left( \left( \mathcal{A}^{\#}\right) _{x},G\right) ,
\end{equation*}%
for any $G\in \mathbf{K}$. Therefore, $\mathcal{A}_{x}%
\simeq%
\left( \mathcal{A}^{\#}\right) _{x}$, as desired.
\end{proof}

\section{(Pre)cosheaves}

\subsection{General sites}

We fix a small site $X=\left( \mathbf{C}_{X},Cov\left( X\right) \right) $,
and a category $\mathbf{K}$.

\begin{definition}
\label{Def-(pre)cosheaves}Assume that $\mathbf{K}$ is cocomplete.

\begin{enumerate}
\item A \textbf{precosheaf} $\mathcal{A}$ on $X$ with values in $\mathbf{K}$
is a functor $\mathcal{A}:\mathbf{C}_{X}\rightarrow \mathbf{K}$.

\item A precosheaf $\mathcal{A}$ is \textbf{coseparated} provided%
\begin{equation*}
\mathcal{A}\otimes _{\mathbf{Set}^{\mathbf{C}_{X}}}R\longrightarrow \mathcal{%
A}\otimes _{\mathbf{Set}^{\mathbf{C}_{X}}}h_{U}%
\simeq%
\mathcal{A}\left( U\right)
\end{equation*}%
is an epimorphism for any $U\in \mathbf{C}_{X}$ and for any covering sieve $%
R $ over $U$.

\item A precosheaf $\mathcal{A}$ is a \textbf{cosheaf} provided%
\begin{equation*}
\mathcal{A}\otimes _{\mathbf{Set}^{\mathbf{C}_{X}}}R\longrightarrow \mathcal{%
A}\otimes _{\mathbf{Set}^{\mathbf{C}_{X}}}h_{U}%
\simeq%
\mathcal{A}\left( U\right)
\end{equation*}%
is an isomorphism for any $U\in \mathbf{C}_{X}$ and for any covering sieve $%
R $ over $U$.
\end{enumerate}
\end{definition}

\begin{notation}
\label{Not-(Pre)cosheaves}Let $\mathbf{pCS}\left( X,\mathbf{K}\right) =%
\mathbf{K}^{\mathbf{C}_{X}}$ be the category of precosheaves on $X$ with
values in $\mathbf{K}$, and let $\mathbf{CS}\left( X,\mathbf{K}\right) $ be
the full subcategory of cosheaves.
\end{notation}

\begin{proposition}
Let $G\in \mathbf{K}$, let $\mathcal{A}\in \mathbf{pCS}\left( X,\mathbf{K}%
\right) $, and let $R\subseteq h_{U}$ be a sieve. Then:

\begin{enumerate}
\item \label{Prop-Hom-K-A-ten-Set-CX-G}%
\begin{eqnarray*}
&&Hom_{\mathbf{K}}\left( \mathcal{A}\otimes _{\mathbf{Set}^{\mathbf{C}%
_{X}}}R,G\right) 
\simeq%
Hom_{\mathbf{Set}^{\left( \mathbf{C}_{X}\right) ^{op}}}\left( R,Hom_{\mathbf{%
K}}\left( \mathcal{A},G\right) \right) 
\simeq
\\
&&%
\simeq%
\underset{\left( V\rightarrow U\right) \in \mathbf{C}_{R}}{\underleftarrow{%
\lim }}Hom_{\mathbf{K}}\left( \mathcal{A}\left( V\right) ,G\right) 
\simeq%
Hom_{\mathbf{K}}\left( \underset{\left( V\rightarrow U\right) \in \mathbf{C}%
_{R}}{\underrightarrow{\lim }}\mathcal{A}\left( V\right) ,G\right)
\end{eqnarray*}%
naturally on $G$, $\mathcal{A}$ and $R$.

\item \label{Prop-A-ten-Set-CX-R}%
\begin{equation*}
\mathcal{A}\otimes _{\mathbf{Set}^{\mathbf{C}_{X}}}R%
\simeq%
\underset{\left( V\rightarrow U\right) \in \mathbf{C}_{R}}{\underrightarrow{%
\lim }}\mathcal{A}\left( V\right) .
\end{equation*}
\end{enumerate}
\end{proposition}

\begin{proof}
~

\begin{enumerate}
\item Follows from Proposition \ref{Prop-Hom-KC-(B-Tensor-Set-G,A)} and \cite%
[Corollary I.3.5]{SGA4-1-MR0354652}.

\item Let $G$ in (\ref{Prop-Hom-K-A-ten-Set-CX-G}) runs over all objects of $%
\mathbf{K}$. It follows from (\ref{Prop-Hom-K-A-ten-Set-CX-G}) that $h^{K}%
\simeq%
h^{L}$ where 
\begin{equation*}
K=\mathcal{A}\otimes _{\mathbf{Set}^{\mathbf{C}_{X}}}R,~L=\underset{\left(
V\rightarrow U\right) \in \mathbf{C}_{R}}{\underrightarrow{\lim }}\mathcal{A}%
\left( V\right) .
\end{equation*}%
Due to Yoneda's lemma, $K%
\simeq%
L$.
\end{enumerate}
\end{proof}

\begin{proposition}
\label{Prop-Pretopology-Cosheaves}Let $X$ be a small site with a pretopology
(Definition \ref{Def-Pretopology}). Then for any sieve $R$ generated by a
cover $\left\{ U_{i}\rightarrow U\right\} _{i\in I}$,%
\begin{equation*}
\mathcal{A}\otimes _{\mathbf{Set}^{\mathbf{C}_{X}}}R%
\simeq%
\underset{\left( V\rightarrow U\right) \in \mathbf{C}_{R}}{\underrightarrow{%
\lim }}\mathcal{A}\left( V\right) 
\simeq%
coker\left( \dcoprod\limits_{i,j\in I}\mathcal{A}\left( U_{i}\times
_{U}U_{j}\right) \rightrightarrows \dcoprod\limits_{i\in I}\mathcal{A}\left(
U_{i}\right) \right) .
\end{equation*}
\end{proposition}

\begin{proof}
Apply $Hom_{\mathbf{K}}\left( \_,G\right) $ when $G$ runs over all objects
of $\mathbf{K}$. Apply then \cite[Proposition I.2.12]{SGA4-1-MR0354652}, to
the presheaf of \textbf{sets} $Hom_{\mathbf{K}}\left( \mathcal{A},G\right) $.
\end{proof}

\begin{definition}
\label{Def-Plus-cosheaf}Let $X$ be a small site, and $\mathbf{K}$ be a
category (cocomplete and closed under cofiltered limits). For a precosheaf $%
\mathcal{A}$, define a precosheaf $\left( \mathcal{A}\right) _{+}^{\mathbf{K}%
}$ (or simply $\mathcal{A}_{+})$ by the following:%
\begin{equation*}
\mathcal{A}_{+}\left( U\right) =\underset{R}{\underleftarrow{\lim }}~%
\underset{\left( V\rightarrow U\right) \in \mathbf{C}_{R}}{\underrightarrow{%
\lim }}\mathcal{A}\left( V\right) =\underset{R}{\underleftarrow{\lim }}%
~\left( \mathcal{A}\otimes _{\mathbf{Set}^{\mathbf{C}_{X}}}R\right) .
\end{equation*}%
where $R\subseteq h^{U}$ runs over all covering sieves over $U$.
\end{definition}

\begin{proposition}
If the topology is induced by a pretopology (Definition \ref{Def-Pretopology}%
), then%
\begin{equation*}
\mathcal{A}_{+}\left( U\right) =\underset{\left\{ U_{i}\rightarrow U\right\} 
}{\underleftarrow{\lim }}~coker\left( \dcoprod\limits_{i,j\in I}\mathcal{A}%
\left( U_{i}\times _{U}U_{j}\right) \rightrightarrows \dcoprod\limits_{i\in
I}\mathcal{A}\left( U_{i}\right) \right) .
\end{equation*}
\end{proposition}

\begin{proof}
Follows from Proposition \ref{Prop-Pretopology-Cosheaves}.
\end{proof}

\begin{lemma}
\label{Lemma-Epi-in-Pro(K)}Let%
\begin{equation*}
f:\mathcal{X}=\left( X_{i}\right) _{i\in \mathbf{I}}\longrightarrow \mathcal{%
Y}=\left( Y_{j}\right) _{j\in \mathbf{J}}
\end{equation*}%
be a morphism in $\mathbf{Pro}\left( \mathbf{K}\right) $. Then $f$ is an
epimorphism iff 
\begin{equation*}
Hom_{\mathbf{Pro}\left( \mathbf{K}\right) }\left( \mathcal{Y},G\right)
\longrightarrow Hom_{\mathbf{Pro}\left( \mathbf{K}\right) }\left( \mathcal{X}%
,G\right)
\end{equation*}%
is injective for any \textbf{rudimentary} (Remark \ref{Rem-Rudimentary})
object%
\begin{equation*}
G\in \mathbf{K\subseteq Pro}\left( \mathbf{K}\right) .
\end{equation*}
\end{lemma}

\begin{proof}
Let 
\begin{equation*}
\mathcal{Z}=\left( Z_{s}\right) _{s\in \mathbf{S}}\in \mathbf{Pro}\left( 
\mathbf{K}\right) .
\end{equation*}%
For any $s\in \mathbf{S}$, the mapping%
\begin{equation*}
Hom_{\mathbf{Pro}\left( \mathbf{K}\right) }\left( \mathcal{Y},Z_{s}\right)
\longrightarrow Hom_{\mathbf{Pro}\left( \mathbf{K}\right) }\left( \mathcal{X}%
,Z_{s}\right)
\end{equation*}%
is injective. It follows that%
\begin{equation*}
Hom_{\mathbf{Pro}\left( \mathbf{K}\right) }\left( \mathcal{Y},\mathcal{Z}%
\right) =\underleftarrow{\lim }_{s\in \mathbf{S}}Hom_{\mathbf{Pro}\left( 
\mathbf{K}\right) }\left( \mathcal{Y},Z_{s}\right) \longrightarrow 
\underleftarrow{\lim }_{s\in \mathbf{S}}Hom_{\mathbf{Pro}\left( \mathbf{K}%
\right) }\left( \mathcal{X},Z_{s}\right) =Hom_{\mathbf{Pro}\left( \mathbf{K}%
\right) }\left( \mathcal{X},\mathcal{Z}\right)
\end{equation*}%
is injective as well, therefore $f$ is an epimorphism.
\end{proof}

\begin{corollary}
\label{Cor-Pro(K)-K-(co)complete-cosheaf}Assume $\mathbf{K}$ is cocomplete.

\begin{enumerate}
\item A morphism $f:G\rightarrow H$ in $\mathbf{K}$ is an epimorphism iff it
is an epimorphism in $\mathbf{Pro}\left( \mathbf{K}\right) $.

\item Let $\mathcal{A}\in \mathbf{pCS}\left( X,\mathbf{K}\right) $. Then $%
\mathcal{A}$ is coseparated (a cosheaf) iff it is coseparated (a cosheaf)
when considered as a precosheaf with values in $\mathbf{Pro}\left( \mathbf{K}%
\right) $.
\end{enumerate}
\end{corollary}

\begin{proof}
The full inclusion $\mathbf{K}\subseteq \mathbf{Pro}\left( \mathbf{K}\right) 
$ commutes with colimits \cite[dual to Corollary 6.1.17]%
{Kashiwara-Categories-MR2182076}.
\end{proof}

\begin{proposition}
\label{Prop-Hom(K,G)}Let $\mathbf{K}$ be a cocomplete category, let $\mathbf{%
L}$\ be either $\mathbf{K}$ or $\mathbf{Pro}\left( \mathbf{K}\right) $, and
let $\mathcal{A}\in \mathbf{pCS}\left( X,\mathbf{L}\right) $ be a precosheaf.

\begin{enumerate}
\item $\mathcal{A}$ is coseparated iff $Hom_{\mathbf{L}}\left( \mathcal{A}%
,G\right) \in \mathbf{pS}\left( X,\mathbf{Set}\right) $ is separated for any 
$G\in \mathbf{K}$.

\item $\mathcal{A}$ is a cosheaf iff $Hom_{\mathbf{L}}\left( \mathcal{A}%
,G\right) \in \mathbf{pS}\left( X,\mathbf{Set}\right) $ is a sheaf for any $%
G\in \mathbf{K}$.
\end{enumerate}
\end{proposition}

\begin{proof}
It follows from Proposition \ref{Prop-Hom-KC-(B-Tensor-Set-G,A)}, that for
any sieve $R\subseteq h_{U}$, 
\begin{equation*}
Hom_{\mathbf{L}}\left( \mathcal{A}\otimes _{\mathbf{Set}^{\mathbf{C}%
_{X}}}R,G\right) 
\simeq%
Hom_{\mathbf{Set}^{\left( \mathbf{C}_{X}\right) ^{op}}}\left( R,Hom_{\mathbf{%
L}}\left( \mathcal{A},G\right) \right) .
\end{equation*}

Consider the diagrams%
\begin{equation*}
\mathcal{A}\otimes _{\mathbf{Set}^{\mathbf{C}_{X}}}R\overset{\varphi }{%
\longrightarrow }\mathcal{A}\otimes _{\mathbf{Set}^{\mathbf{C}_{X}}}h_{U}%
\simeq%
\mathcal{A}\left( U\right)
\end{equation*}%
and%
\begin{equation*}
\begin{diagram} Hom_{\QTR{bf}{L}}\left(\mathcal{A}\left( U\right) ,G\right)
& \rTo_{\varphi_{G}} & Hom_{\QTR{bf}{L}}\left( \mathcal{A}\otimes
_{\QTR{bf}{Set}^{\QTR{bf}{C}_{X}}}R,G\right) \\ \dTo^{\simeq} & &
\dTo_{\simeq} \\ Hom_{\QTR{bf}{L}}\left( \mathcal{A},G\right) \left(
U\right) & \rTo^{\psi_{G}} & Hom_{\QTR{bf}{Set}^{\left(
\QTR{bf}{C}_{X}\right) ^{op}}}\left( R,Hom_{\QTR{bf}{L}}\left(
\mathcal{A},G\right) \right) \\ \end{diagram}
\end{equation*}%
where $U$ runs over objects of $\mathbf{C}_{X}$, and $R$ runs over covering
sieves.

\begin{enumerate}
\item If $\mathbf{L}=\mathbf{K}$, then $\varphi $ is an epimorphism $\iff $ $%
\varphi _{G}$ is a monomorphism for any $G\in \mathbf{K}$ $\iff $ $\psi _{G}$
is a monomorphism for any $G\in \mathbf{K}$ $\iff $ $Hom_{\mathbf{L}}\left( 
\mathcal{A},G\right) $ is a separated presheaf of sets for any $G\in \mathbf{%
K}$.

If $\mathbf{L}=\mathbf{Pro}\left( \mathbf{K}\right) $, then, due to Lemma %
\ref{Lemma-Epi-in-Pro(K)}, $\varphi $ is an epimorphism $\iff $ $\varphi
_{G} $ is a monomorphism for any $G\in \mathbf{K}$ $\iff $ $Hom_{\mathbf{L}%
}\left( \mathcal{A},G\right) $ is a separated presheaf for any $G\in \mathbf{%
K}$.

\item If $\mathbf{L}=\mathbf{K}$, then $\varphi $ is an isomorphism $\iff $ $%
\varphi _{G}$ is an isomorphism for any $G\in \mathbf{K}$ $\iff $ $\psi _{G}$
is an isomorphism for any $G\in \mathbf{K}$ $\iff $ $Hom_{\mathbf{L}}\left( 
\mathcal{A},G\right) $ is a sheaf of sets for any $G\in \mathbf{K}$.

If $\mathbf{L}=\mathbf{Pro}\left( \mathbf{K}\right) $, then, since $\mathbf{%
Pro}\left( \mathbf{K}\right) ^{op}$ is a full subcategory of $\mathbf{Set}^{%
\mathbf{K}}$, $\varphi $ is an isomorphism $\iff $ $\varphi _{G}$ is an
isomorphism for any $G\in \mathbf{K}$ $\iff $ $Hom_{\mathbf{L}}\left( 
\mathcal{A},G\right) $ is a sheaf for any $G\in \mathbf{K}$.
\end{enumerate}
\end{proof}

\begin{proposition}
\label{Prop-Hom(Pro(K),G)}Assume that $\mathbf{K}$ is cocomplete. Let $%
\mathcal{A}\in \mathbf{pCS}\left( X,\mathbf{Pro}\left( \mathbf{K}\right)
\right) $, and $G\in \mathbf{K}$. Then there is a natural (on $\mathcal{A}$
and $G$) isomorphism%
\begin{equation*}
Hom_{\mathbf{Pro}\left( \mathbf{K}\right) }\left( \left( \mathcal{A}\right)
_{+}^{\mathbf{Pro}\left( \mathbf{K}\right) },G\right) 
\simeq%
\left( Hom_{\mathbf{Pro}\left( \mathbf{K}\right) }\left( \mathcal{A}%
,G\right) \right) _{\mathbf{Set}}^{+}.
\end{equation*}
\end{proposition}

\begin{proof}
The functor 
\begin{equation*}
Hom_{\mathbf{Pro}\left( \mathbf{K}\right) }\left( \_,G\right) :\mathbf{Pro}%
\left( \mathbf{K}\right) \longrightarrow \mathbf{Set}^{op}
\end{equation*}%
commutes with small colimits \cite[dual to Corollary 6.1.17]%
{Kashiwara-Categories-MR2182076}, and cofiltered limits \cite[dual to
Theorem 6.1.8]{Kashiwara-Categories-MR2182076}.
\end{proof}

\begin{theorem}
\label{Th-Properties-of-Plus-Cosheaves}Assume $\mathbf{K}$ is cocomplete.
Let 
\begin{equation*}
\lambda \left( \mathcal{A}\right) :\mathcal{A}_{+}=\left( \mathcal{A}\right)
_{+}^{\mathbf{Pro}\left( \mathbf{K}\right) }\longrightarrow \mathcal{A}
\end{equation*}
be the canonical morphism of functors%
\begin{equation*}
\left( {}\right) _{+}\longrightarrow \mathbf{1}_{\mathbf{pCS}\left( X,%
\mathbf{Pro}\left( \mathbf{K}\right) \right) }:\mathbf{pCS}\left( X,\mathbf{%
Pro}\left( \mathbf{K}\right) \right) \longrightarrow \mathbf{pCS}\left( X,%
\mathbf{Pro}\left( \mathbf{K}\right) \right) .
\end{equation*}%
Then:

\begin{enumerate}
\item The functor $\left( {}\right) _{+}$ is right exact.

\item For any $\mathcal{A}$, $\mathcal{A}_{+}$ is a coseparated precosheaf.

\item A presheaf $\mathcal{A}$ is coseparated iff $\lambda \left( \mathcal{A}%
\right) $ is an epimorphism. In that case $\mathcal{A}_{+}$ is a cosheaf.

\item The following conditions are equivalent:

\begin{enumerate}
\item $\lambda \left( \mathcal{A}\right) $ is an isomorphism.

\item $\mathcal{A}$ is a cosheaf.
\end{enumerate}

\item The functor $\left( {}\right) _{\#}^{\mathbf{Pro}\left( \mathbf{K}%
\right) }=\left( {}\right) _{++}^{\mathbf{Pro}\left( \mathbf{K}\right) }$ is
right adjoint to the inclusion%
\begin{equation*}
i_{X,\mathbf{Pro}\left( \mathbf{K}\right) }:\mathbf{pCS}\left( X,\mathbf{Pro}%
\left( \mathbf{K}\right) \right) \hookrightarrow \mathbf{pCS}\left( X,%
\mathbf{Pro}\left( \mathbf{K}\right) \right) .
\end{equation*}
\end{enumerate}
\end{theorem}

\begin{proof}
Let $G$ run over objects of $\mathbf{K}$ (not of $\mathbf{Pro}\left( \mathbf{%
K}\right) $).

\begin{enumerate}
\item The functor $\mathcal{A}\mapsto \mathcal{A}_{+}$ is the composition of
a colimit $\underset{\left( V\rightarrow U\right) \in \mathbf{C}_{R}}{%
\underrightarrow{\lim }}\mathcal{A}\left( V\right) $ which commutes with
arbitrary colimits, and a codirected limit $\underset{R\subseteq h_{U}}{%
\underleftarrow{\lim }}$ which commutes with \textbf{finite} colimits \cite[%
dual to Proposition 6.1.19]{Kashiwara-Categories-MR2182076}. Therefore, $%
\left( {}\right) _{+}$ is right exact (commutes with finite colimits).

\item Due to Proposition \ref{Prop-Hom(Pro(K),G)},%
\begin{equation*}
Hom_{\mathbf{Pro}\left( \mathbf{K}\right) }\left( \mathcal{A}_{+},G\right) 
\simeq%
\left( Hom_{\mathbf{Pro}\left( \mathbf{K}\right) }\left( \mathcal{A}%
,G\right) \right) ^{+}.
\end{equation*}%
Due to \cite[Proposition II.3.2]{SGA4-1-MR0354652}, $Hom_{\mathbf{K}}\left( 
\mathcal{A}_{+},G\right) $ is separated for any $G\in \mathbf{K}$. Apply
Proposition \ref{Prop-Hom(K,G)}.

\item Due to \cite[Proposition II.3.2]{SGA4-1-MR0354652},%
\begin{equation*}
Hom_{\mathbf{Pro}\left( \mathbf{K}\right) }\left( \mathcal{A},G\right)
\longrightarrow Hom_{\mathbf{Pro}\left( \mathbf{K}\right) }\left( \mathcal{A}%
_{+},G\right)
\end{equation*}%
is a monomorphism iff $Hom_{\mathbf{Pro}\left( \mathbf{K}\right) }\left( 
\mathcal{A},G\right) $ is separated. In that case $Hom_{\mathbf{Pro}\left( 
\mathbf{K}\right) }\left( \mathcal{A}_{+},G\right) $ is a sheaf. Apply
Proposition \ref{Prop-Hom(K,G)}.

\item Due to \cite[Proposition II.3.2]{SGA4-1-MR0354652}, 
\begin{equation*}
Hom_{\mathbf{Pro}\left( \mathbf{K}\right) }\left( \mathcal{A},G\right)
\longrightarrow Hom_{\mathbf{Pro}\left( \mathbf{K}\right) }\left( \mathcal{A}%
_{+},G\right)
\end{equation*}%
is an isomorphism iff $Hom_{\mathbf{Pro}\left( \mathbf{K}\right) }\left( 
\mathcal{A},G\right) $ is a sheaf for any $G\in \mathbf{K}$. Apply
Proposition \ref{Prop-Hom(K,G)}.

\item We need to prove that for any cosheaf $\mathcal{B}$, any morphism $%
\mathcal{B\rightarrow A}$ has a unique decomposition%
\begin{equation*}
\mathcal{B\longrightarrow A}_{++}\longrightarrow \mathcal{A}.
\end{equation*}%
The existence is easy: since $\mathcal{B}_{++}\rightarrow \mathcal{B}$ is an
isomorphism, take the decomposition%
\begin{equation*}
\mathcal{B}%
\simeq%
\mathcal{B}_{++}\longrightarrow \mathcal{A}_{++}\longrightarrow \mathcal{A}.
\end{equation*}%
To prove uniqueness, consider two decompositions%
\begin{equation*}
\begin{diagram} \mathcal{B} & \pile{\rTo^{\alpha} \\ \rTo_{\beta}} &
\mathcal{A}_{++} & \rTo & \mathcal{A} \\ \end{diagram}
\end{equation*}%
and apply $Hom_{\mathbf{Pro}\left( \mathbf{K}\right) }\left( \_,G\right) $:%
\begin{equation*}
\begin{diagram}
Hom_{\QTR{bf}{Pro\left({K}\right)}}\left(\mathcal{A},G\right) & \rTo &
Hom_{\QTR{bf}{Pro\left({K}\right)}}\left(\mathcal{A},G\right) ^{++} &
\pile{\rTo^{Hom_{\QTR{bf}{Pro\left({K}\right)}}\left(\alpha,G \right) } \\
\rTo_{Hom_{\QTR{bf}{Pro\left({K}\right)}}\left(\beta,G \right)}} &
Hom_{\QTR{bf}{Pro\left({K}\right)}}\left(\mathcal{B},G \right). \\
\end{diagram}
\end{equation*}%
It follows that $Hom_{\mathbf{Pro}\left( \mathbf{K}\right) }\left( \alpha
,G\right) =Hom_{\mathbf{Pro}\left( \mathbf{K}\right) }\left( \beta ,G\right) 
$ for any $G\in \mathbf{K}$, therefore $\alpha =\beta $, because $\mathbf{Pro%
}\left( \mathbf{K}\right) ^{op}$ is a full subcategory of $\mathbf{Set}^{%
\mathbf{K}}$.
\end{enumerate}
\end{proof}

\subsection{Topological spaces}

Throughout this subsection, $X$ is a topological space considered as the
site $OPEN\left( X\right) $ (see Example \ref{Site-TOP} and Remark \ref%
{Denote-standard-site-simply}), and $\mathbf{K}$ is a cocomplete category.

\begin{proposition}
\label{Prop-Empty-value}Let $\mathcal{A}$ be a cosheaf with values in $%
\mathbf{K}$. Then $\mathcal{A}\left( \varnothing \right) $ is an initial
object in $\mathbf{K}$.
\end{proposition}

\begin{proof}
Let $\left\{ U_{i}\rightarrow \varnothing \right\} _{i\in I}$ be the empty
covering, i.e. the set of indices $I$ is empty. It is clear that%
\begin{equation*}
Y=\dcoprod\limits_{i\in I}\mathcal{A}\left( U_{i}\right)
\end{equation*}%
is an initial object in $\mathbf{C}$.

If $\mathcal{A}$ is a cosheaf, then%
\begin{equation*}
\mathcal{A}\left( \varnothing \right) =coker\left( \dcoprod\limits_{i\in
\varnothing }\mathcal{A}\left( U_{i}\right) \rightrightarrows
\dcoprod\limits_{\left( i,j\right) \in \varnothing }\mathcal{A}\left(
U_{i}\cap U_{j}\right) \right) =coker\left( Y\rightrightarrows Y\right) 
\simeq%
Y.
\end{equation*}
\end{proof}

\begin{corollary}
\label{Corr-Empty-value}If, in the conditions of Proposition \ref%
{Prop-Empty-value}, $\mathbf{K}$ is $\mathbf{Set}$ or $\mathbf{Pro}\left( 
\mathbf{Set}\right) $, then $\mathcal{A}\left( \varnothing \right)
=\varnothing $. If $\mathbf{K}$ is $\mathbf{Ab}$ or $\mathbf{Pro}\left( 
\mathbf{Ab}\right) $, then $\mathcal{A}\left( \varnothing \right) =0$.
\end{corollary}

\begin{corollary}
\label{AB-cosheaf-is-not-SET-cosheaf}A cosheaf with values in $\mathbf{Ab}$
or $\mathbf{Pro}\left( \mathbf{Ab}\right) $ \textbf{is never a cosheaf} when
considered as a precosheaf with values in $\mathbf{Set}$ or $\mathbf{Pro}%
\left( \mathbf{Set}\right) $.
\end{corollary}

\begin{definition}
\label{constant}\label{constant-AB}Let $G\in \mathbf{K}$. We denote by the
same letter $G$ the following \textbf{constant} precosheaf on $X$ with
values in $\mathbf{K}$ or $\mathbf{Pro}\left( \mathbf{K}\right) $: $G\left(
U\right) 
{:=}%
G$ for all open subsets $U$.
\end{definition}

To introduce local isomorphisms, one needs the notion of a \emph{costalk},
which is dual to the notion of a stalk (Definition \ref{Def-Stalk}) in sheaf
theory.

\begin{definition}
\label{Costalk}\label{Def-Costalk}Assume that a category $\mathbf{K}$ admits
cofiltered limits. Let $\mathcal{A}$ be a precosheaf with values in $\mathbf{%
K}$, and let $x\in X$. The \textbf{costalk} of $\mathcal{A}$ at $x$ is 
\begin{equation*}
\mathcal{A}^{x}%
{:=}%
\underleftarrow{\lim }_{U\in J\left( x\right) }\mathcal{A}\left( U\right)
\end{equation*}%
where $J\left( x\right) $ is the family of open neighborhoods of $x$.
\end{definition}

\begin{remark}
\label{Rem-Costalk-different-categories}In a situation when $K\subseteq L$
is a subcategory, and $\mathcal{A}\in \mathbf{pS}\left( X,\mathbf{K}\right) $
we will use notations $\left( \mathcal{A}\right) _{\mathbf{K}}^{x}$ and $%
\left( \mathcal{A}\right) _{\mathbf{L}}^{x}$ depending on whether the limit
is taken in the category $\mathbf{K}$ or in the category $\mathbf{L}$.
\end{remark}

\begin{example}
Let%
\begin{eqnarray*}
\mathbf{K} &\subseteq &\mathbf{Pro}\left( \mathbf{K}\right) \mathbf{,} \\
\mathcal{A} &\in &\mathbf{pCS}\left( X,\mathbf{K}\right) \subseteq \mathbf{%
pCS}\left( X,\mathbf{Pro}\left( \mathbf{K}\right) \right) .
\end{eqnarray*}%
Then $\left( \mathcal{A}\right) _{\mathbf{K}}^{x}$ is just the limit%
\begin{equation*}
\left( \mathcal{A}\right) _{\mathbf{K}}^{x}=\underleftarrow{\lim }_{U\in
J\left( x\right) }\mathcal{A}\left( U\right) ,
\end{equation*}%
while $\left( \mathcal{A}\right) _{\mathbf{Pro}\left( \mathbf{K}\right)
}^{x} $ is the pro-object represented by the cofiltered diagram%
\begin{equation*}
\left( \mathcal{A}\right) _{\mathbf{Pro}\left( \mathbf{K}\right)
}^{x}=\left( \mathcal{A}\left( U\right) \right) _{U\in J\left( x\right) }.
\end{equation*}
\end{example}

\begin{definition}
\label{Def-Local-isomorphism}Let $\mathbf{K}$ admit cofiltered limits, and
let $f:\mathcal{A\rightarrow B}$ be a morphism in the category of
precosheaves $\mathbf{pCS}\left( X,\mathbf{K}\right) $. We say that $f$ is a 
\textbf{local isomorphism} iff $f^{x}:\mathcal{A}^{x}\rightarrow \mathcal{B}%
^{x}$ is an isomorphism for any $x\in X$. In a situation when $\mathbf{K}%
\subseteq \mathbf{L}$, and%
\begin{equation*}
\mathcal{A},\mathcal{B}\in \mathbf{pCS}\left( X,\mathbf{K}\right) \subseteq 
\mathbf{pCS}\left( X,\mathbf{L}\right) ,
\end{equation*}%
we will say that $f$ is $\mathbf{K}$-local (respectively $\mathbf{L}$-local)
isomorphism iff $\left( f\right) _{\mathbf{K}}^{x}:\left( \mathcal{A}\right)
_{\mathbf{K}}^{x}\rightarrow \left( \mathcal{B}\right) _{\mathbf{K}}^{x}$
(respectively $\left( f\right) _{\mathbf{L}}^{x}:\left( \mathcal{A}\right) _{%
\mathbf{L}}^{x}\rightarrow \left( \mathcal{B}\right) _{\mathbf{L}}^{x}$) is
an isomorphism for any $x\in X$.
\end{definition}

\begin{proposition}
\label{Prop-Cosheafification-is-local-iso}Let $\mathbf{K}$ be a cocomplete
category admitting cofiltered limits. Assume that $\mathbf{CS}\left( X,%
\mathbf{K}\right) \subseteq \mathbf{pCS}\left( X,\mathbf{K}\right) $ is
coreflective, and the coreflection is given by the functor%
\begin{equation*}
\left( {}\right) _{\#}:\mathbf{pCS}\left( X,\mathbf{K}\right)
\longrightarrow \mathbf{CS}\left( X,\mathbf{K}\right) .
\end{equation*}%
Then for any precosheaf $\mathcal{A}$, the natural morphism $\mathcal{A}%
_{\#}\rightarrow \mathcal{A}$ is a local isomorphism.
\end{proposition}

\begin{proof}
Let $x\in X$, and $G\in \mathbf{K}$. Denote by $\mathcal{P}_{x,G}$ the
following \textbf{pointed} precosheaf: $\mathcal{P}_{x,G}\left( U\right) $
is an initial object $J$ when $x\not\in U$, and $\mathcal{P}_{x,G}\left(
U\right) =G$ when $x\in U$. It is easy to check that $\mathcal{P}_{x,G}$ is
in fact a \textbf{cosheaf}, and that for any precosheaf $\mathcal{C}$,%
\begin{equation*}
Hom_{\mathbf{pCS}\left( X,\mathbf{K}\right) }\left( \mathcal{P}_{x,G},%
\mathcal{C}\right) 
\simeq%
\underleftarrow{\lim }_{U\in J\left( x\right) }Hom_{\mathbf{K}}\left( G,%
\mathcal{C}\left( U\right) \right) 
\simeq%
Hom_{\mathbf{K}}\left( G,\mathcal{C}^{x}\right) ,
\end{equation*}%
naturally on $G$ and $\mathcal{C}$. Using the adjointness isomorphism, one
gets%
\begin{equation*}
Hom_{\mathbf{K}}\left( G,\mathcal{A}^{x}\right) 
\simeq%
Hom_{\mathbf{pCS}\left( X,\mathbf{K}\right) }\left( \mathcal{P}_{x,G},%
\mathcal{A}\right) 
\simeq%
Hom_{\mathbf{CS}\left( X,\mathbf{K}\right) }\left( \mathcal{P}_{x,G},%
\mathcal{A}_{\#}\right) 
\simeq%
Hom_{\mathbf{K}}\left( G,\left( \mathcal{A}_{\#}\right) ^{x}\right) ,
\end{equation*}%
for any $G\in \mathbf{K}$. Therefore, $\mathcal{A}^{x}%
\simeq%
\left( \mathcal{A}_{\#}\right) ^{x}$, as desired.
\end{proof}

\begin{example}
Let $\mathcal{A}$ is a precosheaf of abelian groups on $X$. According to 
\cite{Bredon-MR0226631}, Section 2, or \cite{Bredon-Book-MR1481706},
Definition V.12.1, $\mathcal{A}$ is called \emph{locally zero} iff for any $%
x\in X$ and any open neighborhood $U$ of $x$ there exists another open
neighborhood $V$, $x\in V\subseteq U$, such that $\mathcal{A}\left( V\right)
\rightarrow \mathcal{A}\left( U\right) $ is zero. If we consider, however,
the precosheaf $\mathcal{A}$ as a precosheaf of abelian pro-groups, then $%
\mathcal{A}$ is locally zero iff for any $x\in X$, $\mathcal{A}^{x}$ is the
zero object in the category $\mathbf{Pro}\left( \mathbf{Ab}\right) $.
\end{example}

\begin{definition}
A precosheaf $\mathcal{A}$ of abelian (pro-)groups on $X$ is called \textbf{%
locally zero} if $\left( \mathcal{A}\right) _{\mathbf{Pro}\left( \mathbf{Ab}%
\right) }^{x}=0$ for any $x\in X$.
\end{definition}

\begin{definition}
\label{local-isomorphism}\label{Def-Local-isomorphism-Bredon}Let $\mathcal{%
A\rightarrow B}$ be a morphism of precosheaves (with values in $\mathbf{K}$
or $\mathbf{Pro}\left( \mathbf{K}\right) $) on $X$. It is called a \textbf{%
local isomorphism in the sense of Bredon} (shorty: \textbf{strong local
isomorphism}) iff $\left( \mathcal{A}\right) _{\mathbf{Pro}\left( \mathbf{K}%
\right) }^{x}\rightarrow \left( \mathcal{B}\right) _{\mathbf{Pro}\left( 
\mathbf{K}\right) }^{x}$ is an isomorphism for any $x\in X$.
\end{definition}

\begin{remark}
Strong local isomorphisms are local isomorphisms. Indeed, it follows from 
\cite[dual to Proposition 6.3.1]{Kashiwara-Categories-MR2182076}, that if $%
\left( A_{i}\right) _{i\in \mathbf{I}}\rightarrow \left( B_{j}\right) _{J\in 
\mathbf{J}}$ is an isomorphism in $\mathbf{Pro}\left( \mathbf{K}\right) $,
then 
\begin{equation*}
\underleftarrow{\lim }_{i\in \mathbf{I}}A_{i}\longrightarrow \underleftarrow{%
\lim }_{j\in \mathbf{J}}B_{j}
\end{equation*}%
is an isomorphism in $\mathbf{K}$.
\end{remark}

\begin{proposition}
\label{Local-isomorphism-Bredon}Let $f:\mathcal{A\rightarrow B}$ be a
morphism of precosheaves on $X$ with values in $\mathbf{Ab}$ or $\mathbf{Pro}%
\left( \mathbf{Ab}\right) $. Then $f$ is a strong local isomorphism iff both 
$\ker \left( f\right) $ and $coker\left( f\right) $ are locally zero.
\end{proposition}

\begin{proof}
Since cofiltered limits are exact in $\mathbf{Pro}\left( \mathbf{Ab}\right) $
\cite[dual to Proposition 6.1.19]{Kashiwara-Categories-MR2182076}, the
sequence%
\begin{equation*}
\left( \ker \left( f\right) \right) _{\mathbf{Pro}\left( \mathbf{Ab}\right)
}^{x}\longrightarrow \left( \mathcal{A}\right) _{\mathbf{Pro}\left( \mathbf{%
Ab}\right) }^{x}\overset{f^{x}}{\longrightarrow }\left( \mathcal{B}\right) _{%
\mathbf{Pro}\left( \mathbf{Ab}\right) }^{x}\longrightarrow \left(
coker\left( f\right) \right) _{\mathbf{Pro}\left( \mathbf{Ab}\right) }^{x}
\end{equation*}%
is exact. Since $\mathbf{Pro}\left( \mathbf{Ab}\right) $ is an abelian
category \cite[Chapter 8.6]{Kashiwara-Categories-MR2182076}, $f^{x}$ is an
isomorphism iff both $\left( \ker \left( f\right) \right) _{\mathbf{Pro}%
\left( \mathbf{Ab}\right) }^{x}$ and $\left( coker\left( f\right) \right) _{%
\mathbf{Pro}\left( \mathbf{Ab}\right) }^{x}$ are zero.
\end{proof}

\begin{remark}
It follows from Proposition \ref{Local-isomorphism-Bredon} that a morphism $%
f:\mathcal{A}\rightarrow \mathcal{B}$ of precosheaves of abelian groups is a
local isomorphism in the sense of \cite[Section 3]{Bredon-MR0226631}, or 
\cite[Definition V.12.2]{Bredon-Book-MR1481706}, iff it is a strong local
isomorphism in our sense.
\end{remark}

\begin{proposition}
\label{Prop-Strong-local-equivalence}Let $\mathbf{K}$ be cocomplete, and let%
\begin{equation*}
\left( f:\mathcal{A}\longrightarrow \mathcal{B}\right) \in Hom_{\mathbf{pCS}%
\left( X,\mathbf{Pro}\left( \mathbf{K}\right) \right) }\left( \mathcal{A},%
\mathcal{B}\right) .
\end{equation*}%
Then $f$ is a strong local equivalence iff 
\begin{equation*}
Hom_{\mathbf{Pro}\left( \mathbf{K}\right) }\left( \mathcal{B},G\right)
\longrightarrow Hom_{\mathbf{Pro}\left( \mathbf{K}\right) }\left( \mathcal{A}%
,G\right)
\end{equation*}%
is a local equivalence of $\mathbf{Set}$-valued presheaves for all $G\in 
\mathbf{K}$.
\end{proposition}

\begin{proof}
\begin{equation*}
Hom_{\mathbf{Pro}\left( \mathbf{K}\right) }\left( \_,G\right) :\mathbf{Pro}%
\left( \mathbf{K}\right) \longrightarrow \mathbf{Set}
\end{equation*}%
converts cofiltered limits into filtered colimits \cite[dual to Corollary
6.1.17]{Kashiwara-Categories-MR2182076}.
\end{proof}

\section{Examples}

Below is a series of examples of various (pre)cosheaves with values anywhere.

\subsection{Cosheaves}

\begin{example}
\label{Singular-precosheaf}Let $A$ be an abelian group, and let $\Sigma
_{n}\left( \_,A\right) $ be a precosheaf that assigns to $U$ the colimit of
the following sequence:%
\begin{equation*}
S_{n}\left( U,A\right) \overset{\mathbf{ba}}{\longrightarrow }S_{n}\left(
U,A\right) \overset{\mathbf{ba}}{\longrightarrow }S_{n}\left( U,A\right) 
\overset{\mathbf{ba}}{\longrightarrow }...
\end{equation*}%
where $S_{n}\left( U,A\right) $ is the group of singular $A$-valued $n$%
-chains on $U$, and $\mathbf{ba}$ is the barycentric subdivision. It is
proved in \cite{Bredon-MR0226631}, Section 10, and \cite%
{Bredon-Book-MR1481706}, Proposition VI.12.1, that $\Sigma _{n}\left(
\_,A\right) $ is a cosheaf of abelian groups (and of abelian pro-groups, due
to Theorem \ref{Our-cosheaves-vs-Bredon}).
\end{example}

\begin{example}
\label{p0-is-cosheaf}Let $\pi _{0}$ be a precosheaf of sets that assigns to $%
U$ the set $\pi _{0}\left( U\right) $ of path-connected components of $U$.
Then $\pi _{0}$ is a cosheaf of sets (and of pro-sets, due to Theorem \ref%
{Our-cosheaves-vs-Bredon}). This cosheaf is constant if $X$ is locally
path-connected, and is not constant in general. Indeed, $\pi _{0}$ is
clearly coseparated. Let $\left\{ U_{i}\rightarrow U\right\} _{i\in I}$ be
an open covering, and let $P\in U_{s}$ and $Q\in U_{t}$ be two points lying
in the same path-connected component. Therefore, there exists a continuous
path $g:\left[ 0,1\right] \longrightarrow U$ with $g\left( 0\right) =P$ and $%
g\left( 1\right) =Q$. Using Lebesgue's Number Lemma, one proves that $P$ and 
$Q$ define equal elements of the cokernel below. Therefore, the mapping%
\begin{equation*}
coker\left( \dcoprod\limits_{i,j}\pi _{0}\left( U_{i}\cap U_{j}\right)
\rightrightarrows \dcoprod\limits_{i}\pi _{0}\left( U_{i}\right) \right)
\longrightarrow \pi _{0}\left( U\right)
\end{equation*}%
is injective, thus bijective, and $\pi _{0}$ is a cosheaf.
\end{example}

\begin{example}
\label{H0-is-cosheaf-in-AB}Let $A$ be an abelian group, and let $%
H_{0}^{S}\left( \_,A\right) $ be the precosheaf of abelian groups that
assigns to $U$ the zeroth singular homology group $H_{0}^{S}\left(
X,A\right) $. Then $H_{0}^{S}\left( \_,A\right) $ is a cosheaf. Indeed,%
\begin{equation*}
H_{0}^{S}=coker\left( \Sigma _{1}\left( \_,A\right) \longrightarrow \Sigma
_{0}\left( \_,A\right) \right)
\end{equation*}%
where $\Sigma _{n}\left( \_,A\right) $ is the cosheaf from Example \ref%
{Singular-precosheaf}. The embedding%
\begin{equation*}
\mathbf{CS}\left( X,\mathbf{Pro}\left( \mathbf{Ab}\right) \right)
\longrightarrow \mathbf{pCS}\left( X,\mathbf{Pro}\left( \mathbf{Ab}\right)
\right) ,
\end{equation*}%
being left adjoint to $\left( {}\right) _{\#}$, commutes with colimits.
Therefore, $H_{0}^{S}\left( \_,A\right) $ is a cosheaf because $\Sigma
_{1}\left( \_,A\right) $ and $\Sigma _{0}\left( \_,A\right) $ are cosheaves. 
$H_{0}^{S}\left( \_,A\right) $ is constant if $X$ is locally path-connected.
However, $H_{0}^{S}\left( \_,A\right) $ is not constant in general, see
Example \ref{Converging-sequence}.
\end{example}

\begin{example}
\label{Ex-Fundamental-groupoid}Let $\mathbf{Gpd}$ be the category of small
groupoids. Consider the following precosheaf $\Pi _{1}\in \mathbf{pCS}\left(
X,\mathbf{Gpd}\right) $: for an open subset $U\subseteq X$ let $\Pi
_{1}\left( U\right) $ be the fundamental groupoid of $U$. Then, due to the
main theorem in \cite%
{Brown-Salleh-1984-A-van-Kampen-theorem-for-unions-on-nonconnected-spaces-MR751476}%
, for any open covering $\left( U_{i}\right) _{i\in I}$ of $U$, the morphism%
\begin{equation*}
coker\left( \dcoprod\limits_{i,j\in I}\Pi _{1}\left( U_{i}\cap U_{j}\right)
\rightrightarrows \dcoprod\limits_{i\in I}\Pi _{1}\left( U_{i}\right)
\right) \longrightarrow \Pi _{1}\left( U\right)
\end{equation*}%
is an isomorphism of groupoids. Therefore, $\Pi _{1}$ is a cosheaf of
groupoids.
\end{example}

\subsection{Precosheaves}

\begin{example}
\label{Non-smooth-precosheaf}Let $X$ be the closed interval $\left[ 0,1%
\right] $, and let $\mathcal{A}$ assign to $U$ the group $S_{1}\left( U,%
\mathbb{Z}\right) $ of singular $1$-chains on $U$. It is proved in \cite[%
Remark 5.9]{Bredon-MR0226631}, and \cite[Example VI.5.9]%
{Bredon-Book-MR1481706}, that this precosheaf of abelian groups is not
smooth.
\end{example}

\begin{example}
\label{Ex-Non-smooth-precosheaf}Fix $n\geq 1$. Let again $X=I=\left[ 0,1%
\right] $, and let $\mathcal{A}$ assign to $U$ the \textbf{set} $%
Simp_{n}\left( U\right) $ of singular $n$-\textbf{simplices} on $U$, i.e.%
\begin{equation*}
\mathcal{A}\left( U\right) 
{:=}%
Simp_{n}\left( U\right) =U^{\Delta ^{n}}=Hom_{\mathbf{Top}}\left( \Delta
^{n},U\right) .
\end{equation*}%
Then $\mathcal{A}$ is not smooth as a precosheaf of sets. Indeed, let $%
\mathcal{B}=\left( \mathcal{A}\right) _{+}^{\mathbf{Pro}\left( \mathbf{Set}%
\right) }$. For an open $U\subseteq X$,%
\begin{equation*}
\mathcal{B}\left( U\right) =\left( B_{\left\{ U_{i}\right\} }\right)
_{\left\{ U_{i}\right\} }
\end{equation*}%
where $\left\{ U_{i}\right\} $ runs over open covers of $U$, and%
\begin{equation*}
B_{\left\{ U_{i}\right\} }=\left\{ \sigma :\Delta ^{n}\longrightarrow
U~|~\exists i\left( \sigma \left( \Delta ^{n}\right) \subseteq U_{i}\right)
\right\} .
\end{equation*}%
It can be checked that:

\begin{enumerate}
\item $\mathcal{B}$ is a cosheaf of pro-sets.

\item For any $U\neq \varnothing $, the pro-object $\mathcal{B}\left(
U\right) $ is \textbf{not} rudimentary (Remark \ref{Rem-Rudimentary}).
\end{enumerate}

It follows that%
\begin{equation*}
\left( \mathcal{A}\right) _{\#}^{\mathbf{Pro}\left( \mathbf{Set}\right) }%
\simeq%
\left( \mathcal{A}\right) _{++}^{\mathbf{Pro}\left( \mathbf{Set}\right) }%
\simeq%
\left( \mathcal{A}\right) _{+}^{\mathbf{Pro}\left( \mathbf{Set}\right) }%
\simeq%
\mathcal{B},
\end{equation*}%
and this cosheaf does not take values in $\mathbf{Set}$. Therefore, due to
Theorem \ref{Our-cosheaves-vs-Bredon}, $\mathcal{A}$ is \textbf{not smooth}.
However, since $\mathbf{Set}$ is locally presentable (even locally finitely
presentable), there exists, due to Theorem \ref{Th-Main}(\ref%
{Th-Main-K-locally-presentable}), a cosheafification%
\begin{equation*}
\left( {}\right) _{\#}^{\mathbf{Set}}:\mathbf{pCS}\left( X,\mathbf{Set}%
\right) \longrightarrow \mathbf{CS}\left( X,\mathbf{Set}\right) .
\end{equation*}%
It can be checked that $\left( \mathcal{A}\right) _{\#}^{\mathbf{Set}}$ is 
\textbf{rather trivial}: $\left( \mathcal{A}\right) _{\#}^{\mathbf{Set}%
}\left( U\right) =U$, i.e. the result is as if our space $X$ were a \textbf{%
discrete} space. The natural morphism $\left( \mathcal{A}\right) _{\#}^{%
\mathbf{Set}}\rightarrow \mathcal{A}$ sends any point%
\begin{equation*}
a\in \left( \mathcal{A}\right) _{\#}^{\mathbf{Set}}\left( U\right) =U
\end{equation*}%
to the \textbf{constant} ($\sigma \left( t\right) \equiv a$) singular
simplex 
\begin{equation*}
\sigma \in Simp_{n}\left( U\right) =\mathcal{A}\left( U\right) .
\end{equation*}%
Let us calculate the costalks:%
\begin{eqnarray*}
\left( \mathcal{A}\right) _{\mathbf{Set}}^{x} &=&\dbigcap\limits_{U\in
J\left( x\right) }Simp_{n}\left( U\right) =\mathbf{pt}, \\
\left( \left( \mathcal{A}\right) _{\#}^{\mathbf{Set}}\right) _{\mathbf{Set}%
}^{x} &=&\dbigcap\limits_{U\in J\left( x\right) }U=\mathbf{pt},
\end{eqnarray*}%
while%
\begin{equation*}
\left( \mathcal{A}\right) _{\mathbf{Pro}\left( \mathbf{Set}\right) }^{x}%
\simeq%
\left( \mathcal{B}\right) _{\mathbf{Pro}\left( \mathbf{Set}\right) }^{x}
\end{equation*}%
are non-rudimentary pro-sets. It is clear that $\left( \mathcal{A}\right)
_{\#}^{\mathbf{Set}}\rightarrow \mathcal{A}$ is a $\mathbf{Set}$-local
isomorphism, but \textbf{not} a strong local isomorphism (because $\mathcal{A%
}$ is not smooth).
\end{example}

\begin{example}
\label{p0-prime-not-cosheaf}Let $\pi $ be a precosheaf of sets that assigns
to $U$ the set $\pi \left( U\right) $ of connected components of $U$. This
precosheaf is coseparated. If $X$ is \textbf{locally connected}, then, for
any open subset $U\subseteq X$, the pro-homotopy set $pro$-$\pi _{0}\left(
U\right) $ is isomorphic to the \textbf{rudimentary} (Remark \ref%
{Rem-Rudimentary}) pro-set $\pi \left( U\right) $. It follows from Theorem %
\ref{Our-cosheaves-vs-Bredon}, that $\pi 
\simeq%
\left( \mathbf{pt}\right) _{\#}$ where $\mathbf{pt}$ is the one-point
constant precosheaf. Therefore, $\mathbf{pt}$ is smooth, and $\pi $ a
constant cosheaf (compare to \cite[Remark 5.11]{Bredon-MR0226631}).

In general, if $X$ is \textbf{not} locally connected, $\pi $ is \textbf{not
a cosheaf}. Indeed, let 
\begin{equation*}
X=Y\cup Z\subseteq \mathbb{R}^{2},
\end{equation*}%
where $Y$ is the line segment between the points $\left( 0,1\right) $ and $%
\left( 0,-1\right) $, and $Y$ is the graph of $y=\sin \left( \frac{1}{x}%
\right) $ for $0<x\leq 2\pi $. Let further%
\begin{eqnarray*}
X &=&U=U_{1}\cup U_{2}, \\
U_{1} &=&\left\{ \left( x,y\right) \in X~|~y>-\frac{1}{2}\right\} , \\
U_{2} &=&\left\{ \left( x,y\right) \in X~|~y<\frac{1}{2}\right\} .
\end{eqnarray*}%
$X$ is a connected (not locally connected!) compact metric space. Take $%
P=\left( 0,1\right) \in U_{1}$ and $Q=\left( \frac{3\pi }{2},-1\right) \in
U_{2}$. Since $U=X$ is connected, these two points are mapped to the same
point of $U$ under the canonical mapping $U_{1}\sqcup U_{2}\longrightarrow U$%
. However, these two points define \textbf{different} elements of the colimit%
\begin{equation*}
coker\left( \pi \left( U_{1}\cap U_{2}\right) \rightrightarrows \pi \left(
U_{1}\right) \sqcup \pi \left( U_{1}\right) \right) .
\end{equation*}%
Therefore,%
\begin{equation*}
coker\left( \pi \left( U_{1}\cap U_{2}\right) \rightrightarrows \pi \left(
U_{1}\right) \sqcup \pi \left( U_{1}\right) \right) \longrightarrow \pi
\left( U\right) =\pi \left( X\right)
\end{equation*}%
is not injective, and $\pi $ is not a cosheaf.

See also Example \ref{Converging-sequence}.
\end{example}

\begin{example}
\label{Converging-sequence}Let $X$ be the following sequence converging to
zero (together with the limit):%
\begin{equation*}
X=\left\{ 0\right\} \cup \left\{ 1,\frac{1}{2},\frac{1}{3},\frac{1}{4}%
,...\right\} \subseteq \mathbb{R}.
\end{equation*}%
The precosheaves $\pi $ and $\pi _{0}$ from Examples \ref%
{p0-prime-not-cosheaf} and \ref{p0-is-cosheaf} coincide on $X$. Therefore, $%
\pi =\pi _{0}$ is a cosheaf. However, it is \textbf{not} constant. To see
this, just compare the costalks at different points $x\in X$: $\left( \pi
\right) _{\mathbf{Pro}\left( \mathbf{Set}\right) }^{x}=\left\{ \mathbf{pt}%
\right\} $ if $x\neq 0$, while $\left( \pi \right) _{\mathbf{Pro}\left( 
\mathbf{Set}\right) }^{0}$ is a \textbf{non-rudimentary} (Remark \ref%
{Rem-Rudimentary}) pro-set. Consider the constant precosheaf $\mathbf{pt}$.
Due to Corollary \ref{Cor-One-point-constant}, $\left( \mathbf{pt}\right)
_{\#}^{\mathbf{Pro}\left( \mathbf{Set}\right) }%
\simeq%
pro$-$\pi _{0}$. The latter cosheaf does \textbf{not} take values in $%
\mathbf{Set}$,\ therefore, due to Theorem \ref{Our-cosheaves-vs-Bredon}, the
precosheaf $\mathbf{pt}$ is \textbf{not} smooth. Similarly, it can be
proved, that the cosheaf $H_{0}^{S}\left( \_,A\right) $ from Example \ref%
{H0-is-cosheaf-in-AB} is not constant on $X$, while the constant precosheaf $%
A$ is not smooth, because $\left( A\right) _{\#}%
\simeq%
pro$-$H_{0}\left( \_,A\right) $ does not take values in $\mathbf{Ab}$.

It appears that the cosheaf $\left( \mathbf{pt}\right) _{\#}^{\mathbf{Set}}$
is \textbf{rather trivial}. Similarly to Example \ref%
{Ex-Non-smooth-precosheaf}, it can be proved that $\left( \mathbf{pt}\right)
_{\#}^{\mathbf{Set}}\left( U\right) =U$, i.e. the result is as if our space $%
X$ were a \textbf{discrete} space.
\end{example}

\section{\label{Section-Proofs}Proofs of the main results}

\subsection{Proof of Theorem \protect\ref{Th-Main} (\protect\ref%
{Th-Main-K-locally-presentable})}

\begin{proof}
The proof goes through the following three steps:

\begin{enumerate}
\item $\mathbf{CS}\left( X,\mathbf{K}\right) $ is the $\mathbf{K}$-valued
model $\mathbf{Mod(}\mathfrak{S}\mathbf{,K)}$ \cite[Definition 2.55 and 2.60]%
{Adamek-Rosicky-1994-Locally-presentable-categories-MR1294136}, of the
following $\underrightarrow{\lim }$-sketch%
\begin{equation*}
\mathfrak{S}\mathbf{=}\left( \mathbf{C}_{X},\mathfrak{L}=\varnothing ,%
\mathfrak{C},\mathbf{K},\sigma \right) .
\end{equation*}%
$\mathfrak{C}$ is the family of diagrams%
\begin{equation*}
\mathbf{C}_{R}\subseteq \mathbf{C}_{U}\longrightarrow \mathbf{C}_{X}
\end{equation*}%
in $\mathbf{C}_{X}$ ($R$ runs over covering sieves over $U$), where $\sigma
\left( R\right) $ is the corresponding cocone%
\begin{equation*}
\sigma \left( R\right) =\left( \mathbf{C}_{R}\mathbf{\hookrightarrow C}%
_{U}\right) .
\end{equation*}%
A precosheaf%
\begin{equation*}
\mathcal{A}\in \mathbf{pCS}\left( X,\mathbf{K}\right) =\mathbf{K}^{\mathbf{C}%
_{X}}
\end{equation*}%
is a cosheaf iff 
\begin{equation*}
\left( \underrightarrow{\lim }_{\left( V\rightarrow U\right) \in \mathbf{C}%
_{R}}\mathcal{A}\left( V\right) \right) \longrightarrow \mathcal{A}\left(
U\right)
\end{equation*}%
is an isomorphism for all $U\in \mathbf{C}_{X}$ and for all sieves $R\in
Cov\left( U\right) $. Therefore, $\mathcal{A}$ is a cosheaf iff $\mathcal{A}$
maps any cocone $\sigma \left( R\right) $ into a $\underrightarrow{\lim }$%
-cocone in $\mathbf{K}$, i.e. $\mathbf{CS}\left( X,\mathbf{K}\right) $ is
indeed the model $\mathbf{Mod(}\mathfrak{S}\mathbf{,K)}$.

\item Due to \cite[Theorem 2.60]%
{Adamek-Rosicky-1994-Locally-presentable-categories-MR1294136}, the category 
$\mathbf{Mod(}\mathfrak{S}\mathbf{,K)}$ is accessible. Since $\mathfrak{S}$
is a $\underrightarrow{\lim }$-sketch ($\mathfrak{L}=\varnothing $), the
category is cocomplete, therefore locally presentable \cite[Corollary 2.47]%
{Adamek-Rosicky-1994-Locally-presentable-categories-MR1294136}. See also 
\cite[Remark 2.63]%
{Adamek-Rosicky-1994-Locally-presentable-categories-MR1294136}.

\item Due to \cite[Theorem 1.58 and Theorem 1.20]%
{Adamek-Rosicky-1994-Locally-presentable-categories-MR1294136}, the category 
$\mathbf{CS}\left( X,\mathbf{K}\right) $, being locally presentable, is
co-wellpowered, and has a generator. The inclusion%
\begin{equation*}
i_{X,\mathbf{K}}:\mathbf{CS}\left( X,\mathbf{K}\right) \hookrightarrow 
\mathbf{pCS}\left( X,\mathbf{K}\right)
\end{equation*}%
clearly preserves direct limits, therefore, due to the dual to \cite[Freyd's
special adjoint functor theorem, Ch. 0.7]%
{Adamek-Rosicky-1994-Locally-presentable-categories-MR1294136}, $i_{X,%
\mathbf{K}}$ is a left adjoint.
\end{enumerate}
\end{proof}

\subsection{Proof of Theorem \protect\ref{Th-Main} (\protect\ref%
{Th-Main-K-op-locally-presentable}-\protect\ref%
{Th-Main-Pro(K)-K-(co)complete})}

\subsubsection{Proof of Theorem \protect\ref{Th-Main} (\protect\ref%
{Th-Main-K-op-locally-presentable})}

\begin{proof}
Let $\mathbf{L}=\mathbf{K}^{op}$. Since%
\begin{equation*}
\mathbf{pCS}\left( X,\mathbf{K}\right) ^{op}%
\simeq%
\mathbf{pS}\left( X,\mathbf{K}^{op}\right) 
\simeq%
\mathbf{pS}\left( X,\mathbf{L}\right)
\end{equation*}%
and%
\begin{equation*}
\mathbf{CS}\left( X,\mathbf{K}\right) ^{op}%
\simeq%
\mathbf{S}\left( X,\mathbf{K}^{op}\right) 
\simeq%
\mathbf{S}\left( X,\mathbf{L}\right) ,
\end{equation*}%
it is enough to apply Theorem \ref{Th-Sheaves-K-locally-presentable}: $%
\mathbf{S}\left( X,\mathbf{L}\right) \subseteq \mathbf{pS}\left( X,\mathbf{L}%
\right) $ is a reflective subcategory.
\end{proof}

\subsubsection{Proof of Theorem \protect\ref{Th-Main} (\protect\ref%
{Th-Main-K-op-locally-finitely-presentable})}

\begin{proof}
Let again $\mathbf{L}=\mathbf{K}^{op}$. Since%
\begin{equation*}
\mathbf{pCS}\left( X,\mathbf{K}\right) ^{op}%
\simeq%
\mathbf{pS}\left( X,\mathbf{K}^{op}\right) 
\simeq%
\mathbf{pS}\left( X,\mathbf{L}\right)
\end{equation*}%
and%
\begin{equation*}
\mathbf{CS}\left( X,\mathbf{K}\right) ^{op}%
\simeq%
\mathbf{S}\left( X,\mathbf{K}^{op}\right) 
\simeq%
\mathbf{S}\left( X,\mathbf{L}\right) ,
\end{equation*}%
it is enough to apply Theorem \ref{Th-Properties-of-Plus}: $\mathbf{S}\left(
X,\mathbf{L}\right) \subseteq \mathbf{pS}\left( X,\mathbf{L}\right) $ is a
reflective subcategory, and a reflection is given by%
\begin{equation*}
\mathcal{A}\longmapsto \left( \mathcal{A}\right) _{\mathbf{K}^{op}}^{++}.
\end{equation*}%
Therefore, 
\begin{equation*}
\mathbf{CS}\left( X,\mathbf{L}\right) \subseteq \mathbf{pCS}\left( X,\mathbf{%
L}\right)
\end{equation*}%
is a coreflective subcategory, and a coreflection is given by%
\begin{equation*}
\mathcal{A}\longmapsto \left( \mathcal{A}\right) _{\mathbf{K}%
^{op}}^{++}=\left( \mathcal{A}\right) _{++}^{\mathbf{K}}.
\end{equation*}
\end{proof}

\subsubsection{Proof of Theorem \protect\ref{Th-Main} (\protect\ref%
{Th-Main-Pro(K)-K-(co)complete})}

\begin{proof}
\textbf{(a)} Follows from Corollary \ref{Cor-Pro(K)-K-(co)complete-cosheaf}.

\textbf{(b)} Follows from Theorem \ref{Th-Properties-of-Plus-Cosheaves}.
\end{proof}

\subsection{Proof of Theorem \protect\ref{Main-local-iso-pro-K}}

\begin{proof}
~

\begin{enumerate}
\item Apply Proposition \ref{Prop-Cosheafification-is-local-iso} to the
category $\mathbf{Pro}\left( \mathbf{K}\right) $.

\item Let 
\begin{equation*}
\left( f:\mathcal{A\rightarrow B}\right) \in Hom_{\mathbf{CS}\left( X,%
\mathbf{Pro}\left( \mathbf{K}\right) \right) }\left( \mathcal{A},\mathcal{B}%
\right)
\end{equation*}%
be a strong local equivalence between cosheaves, and $G$ run over objects of 
$\mathbf{K}$. Due to Proposition \ref{Prop-Strong-local-equivalence} and \ref%
{Prop-Hom(K,G)}, 
\begin{equation*}
Hom_{\mathbf{Pro}\left( \mathbf{K}\right) }\left( f,G\right) :Hom_{\mathbf{%
Pro}\left( \mathbf{K}\right) }\left( \mathcal{B},G\right) \longrightarrow
Hom_{\mathbf{Pro}\left( \mathbf{K}\right) }\left( \mathcal{A},G\right)
\end{equation*}%
is a local isomorphism between sheaves of sets. It is well-known (see, for
example, \cite[Ch. I.1]{Bredon-Book-MR1481706}) that a local isomorphism
between sheaves of sets is an isomorphism, therefore $Hom_{\mathbf{Pro}%
\left( \mathbf{K}\right) }\left( f,G\right) $ is an isomorphism for any $%
G\in \mathbf{K}$. $f$ is then an isomorphism because $\left( \mathbf{Pro}%
\left( \mathbf{K}\right) \right) ^{op}$ is a full subcategory of $\mathbf{Set%
}^{\mathbf{K}}$.

\item It is assumed that the composition%
\begin{equation*}
\mathcal{B}\longrightarrow \left( \mathcal{A}\right) _{\#}^{\mathbf{Pro}%
\left( \mathbf{K}\right) }\longrightarrow \mathcal{A}
\end{equation*}%
of two morphisms is a strong local isomorphism. The second morphism is a
strong local isomorphism, too. Therefore the first morphism is a strong
local isomorphism between cosheaves, thus an isomorphism.
\end{enumerate}
\end{proof}

\subsection{Proof of Theorem \protect\ref{Our-cosheaves-vs-Bredon}}

\begin{proof}
\textbf{(a)} Follows from Theorem \ref{Th-Main} (\ref%
{Th-Main-Pro(K)-K-(co)complete-cosheaf}).

\textbf{(b)} If $\left( \mathcal{A}\right) _{\#}^{\mathbf{Pro}\left( \mathbf{%
K}\right) }$ takes values in $\mathbf{K}$, consider the diagram%
\begin{equation*}
\mathcal{A}\overset{\mathbf{1}_{\mathcal{A}}}{\longrightarrow }\mathcal{A}%
\longleftarrow \left( \mathcal{A}\right) _{\#}^{\mathbf{Pro}\left( \mathbf{K}%
\right) }
\end{equation*}%
of strong local equivalences in $\mathbf{pCS}\left( X,\mathbf{K}\right) $.
The diagram guarantees that $\mathcal{A}$ is smooth.

Conversely, assume that $\mathcal{A}$ is smooth. There exists either a
diagram%
\begin{equation*}
\mathcal{A}\longrightarrow \mathcal{B}\longleftarrow \mathcal{C}
\end{equation*}%
or a diagram 
\begin{equation*}
\mathcal{A\longleftarrow B}^{\prime }\longrightarrow \mathcal{C}
\end{equation*}%
of strong local equivalences with a \textbf{cosheaf} $\mathcal{C}\in \mathbf{%
CS}\left( X,\mathbf{K}\right) $. In the first case, the diagram%
\begin{equation*}
\left( \mathcal{A}\right) _{\#}^{\mathbf{Pro}\left( \mathbf{K}\right)
}\longrightarrow \left( \mathcal{B}\right) _{\#}^{\mathbf{Pro}\left( \mathbf{%
K}\right) }\longleftarrow \mathcal{C}
\end{equation*}%
consists of strong local isomorphisms (therefore isomorphisms, due to
Theorem \ref{Main-local-iso-pro-K}) between cosheaves. It follows that $%
\left( \mathcal{A}\right) _{\#}^{\mathbf{Pro}\left( \mathbf{K}\right) }$
takes values in $\mathbf{K}$, since $\mathcal{C}$ does. In the second case,
the diagram%
\begin{equation*}
\left( \mathcal{A}\right) _{\#}^{\mathbf{Pro}\left( \mathbf{K}\right)
}\longleftarrow \left( \mathcal{B}\right) _{\#}^{\mathbf{Pro}\left( \mathbf{K%
}\right) }\longrightarrow \left( \mathcal{C}\right) _{\#}^{\mathbf{Pro}%
\left( \mathbf{K}\right) }\longrightarrow \mathcal{C}
\end{equation*}%
consists of strong local isomorphisms between cosheaves. It follows again
that $\left( \mathcal{A}\right) _{\#}^{\mathbf{Pro}\left( \mathbf{K}\right)
} $ takes values in $\mathbf{K}$, since $\mathcal{C}$ does.
\end{proof}

\subsection{Proof of Theorem \protect\ref{Main-constant}}

\begin{proposition}
\label{Dual-pro-p0}\label{Dual-pro-H0}Let $\mathbf{K}$ be a cocomplete
category. For any $G,H\in \mathbf{K}$ and any topological space $U$, the set%
\begin{equation*}
Hom_{\mathbf{Pro}\left( \mathbf{K}\right) }\left( G\otimes _{\mathbf{Set}}pro%
\text{-}\pi _{0}\left( U\right) ,H\right)
\end{equation*}%
is naturally (on $G$, $H$ and $U$) isomorphic to the set $Hom_{\mathbf{K}%
}\left( G,H\right) ^{U}$ of continuous functions $U\rightarrow Hom_{\mathbf{K%
}}\left( G,H\right) $ where $Hom_{\mathbf{K}}\left( G,H\right) $ is supplied
with the discrete topology.
\end{proposition}

\begin{proof}
Let $U\rightarrow \left( Y_{j}\right) _{j\in \mathbf{J}}$ be a polyhedral
expansion. Then%
\begin{equation*}
G\otimes _{\mathbf{Set}}pro\text{-}\pi _{0}\left( U\right) =\left( G\otimes
_{\mathbf{Set}}\pi _{0}\left( Y_{j}\right) \right) _{j\in \mathbf{J}}.
\end{equation*}%
Therefore,%
\begin{eqnarray*}
&&Hom_{\mathbf{Pro}\left( \mathbf{K}\right) }\left( G\otimes _{\mathbf{Set}%
}pro\text{-}\pi _{0}\left( U\right) ,H\right) 
\simeq%
\underrightarrow{\lim }_{j\in \mathbf{J}}Hom_{\mathbf{K}}\left( G\otimes _{%
\mathbf{Set}}\pi _{0}\left( Y_{j}\right) ,H\right) 
\simeq
\\
&&%
\simeq%
\underrightarrow{\lim }_{j\in \mathbf{J}}Hom_{\mathbf{Set}}\left( \pi
_{0}\left( Y_{j}\right) ,Hom_{\mathbf{K}}\left( G,H\right) \right) 
\simeq%
\underrightarrow{\lim }_{j\in \mathbf{J}}Hom_{\mathbf{Top}}\left( Y_{j},Hom_{%
\mathbf{K}}\left( G,H\right) \right) 
\simeq
\\
&&%
\simeq%
\underrightarrow{\lim }_{j\in \mathbf{J}}Hom_{H\left( \mathbf{Top}\right)
}\left( Y_{j},Hom_{\mathbf{K}}\left( G,H\right) \right) 
\simeq%
Hom_{H\left( \mathbf{Top}\right) }\left( U,Hom_{\mathbf{K}}\left( G,H\right)
\right) 
\simeq
\\
&&%
\simeq%
Hom_{\mathbf{Top}}\left( U,Hom_{\mathbf{K}}\left( G,H\right) \right) 
\simeq%
Hom_{\mathbf{K}}\left( G,H\right) ^{U}.
\end{eqnarray*}

The bijections 
\begin{eqnarray*}
&&Hom_{\mathbf{Top}}\left( Y_{j},Hom_{\mathbf{K}}\left( G,H\right) \right) 
\simeq%
Hom_{H\left( \mathbf{Top}\right) }\left( Y_{j},Hom_{\mathbf{K}}\left(
G,H\right) \right) , \\
&&Hom_{H\left( \mathbf{Top}\right) }\left( U,Hom_{\mathbf{K}}\left(
G,H\right) \right) 
\simeq%
Hom_{\mathbf{Top}}\left( U,Hom_{\mathbf{K}}\left( G,H\right) \right) ,
\end{eqnarray*}%
above are due to the fact that $Hom_{\mathbf{K}}\left( G,H\right) $ is
discrete, therefore each homotopy class of mappings consists of a single
mapping. The bijection%
\begin{equation*}
\underrightarrow{\lim }_{j\in \mathbf{J}}Hom_{H\left( \mathbf{Top}\right)
}\left( Y_{j},Hom_{\mathbf{K}}\left( G,H\right) \right) 
\simeq%
Hom_{H\left( \mathbf{Top}\right) }\left( U,Hom_{\mathbf{K}}\left( G,H\right)
\right)
\end{equation*}%
follows from the definition of an expansion. Since the spaces $Y_{j}$, being
polyhedra, are locally connected, and $Hom_{\mathbf{K}}\left( G,H\right) $
is discrete, the bijections%
\begin{equation*}
Hom_{\mathbf{Set}}\left( \pi _{0}\left( Y_{j}\right) ,Hom_{\mathbf{K}}\left(
G,H\right) \right) 
\simeq%
Hom_{\mathbf{Top}}\left( Y_{j},Hom_{\mathbf{K}}\left( G,H\right) \right)
\end{equation*}%
follow easily.
\end{proof}

\subsubsection{Proof of the theorem}

\begin{proof}
~

\begin{enumerate}
\item Due to Proposition \ref{Prop-Hom(K,G)} and \ref{Prop-Hom(Pro(K),G)},
it is enough to prove that, for any $H\in \mathbf{K}$, the presheaf of sets%
\begin{equation*}
\mathcal{B}%
{:=}%
Hom_{\mathbf{Pro}\left( \mathbf{K}\right) }\left( G\otimes _{\mathbf{Set}}pro%
\text{-}\pi _{0},H\right)
\end{equation*}%
is a sheaf, and that $\mathcal{C}^{\#}%
\simeq%
\mathcal{B}$ for the constant presheaf of sets%
\begin{equation*}
\mathcal{C}%
{:=}%
Hom_{\mathbf{Pro}\left( \mathbf{K}\right) }\left( G,H\right) =Hom_{\mathbf{K}%
}\left( G,H\right) .
\end{equation*}%
Due to Proposition \ref{Dual-pro-p0}, for any open subset $U$ of $X$,%
\begin{equation*}
\mathcal{B}\left( U\right) =Hom_{\mathbf{Pro}\left( \mathbf{K}\right)
}\left( G\otimes _{\mathbf{Set}}pro\text{-}\pi _{0},H\right) 
\simeq%
Hom_{\mathbf{K}}\left( G,H\right) ^{U}.
\end{equation*}%
For any open covering $\left\{ U_{i}\rightarrow U\right\} _{i\in I}$ the
space $U$ is isomorphic in the category $\mathbf{Top}$ to the cokernel%
\begin{equation*}
coker\left( \dcoprod\limits_{i,j\in I}U_{i}\cap U_{j}\rightrightarrows
\dcoprod\limits_{i\in I}U_{i}\right) ,
\end{equation*}%
therefore%
\begin{eqnarray*}
\mathcal{B}\left( U\right) &=&Hom_{\mathbf{K}}\left( G,H\right) ^{U}%
\simeq%
\ker \left( \dprod\limits_{i\in I}Hom_{\mathbf{K}}\left( G,H\right)
^{U_{i}}\rightrightarrows \dprod\limits_{i,j\in I}Hom_{\mathbf{K}}\left(
G,H\right) ^{U_{i}\cap U_{j}}\right) = \\
&&%
\simeq%
\ker \left( \dprod\limits_{i\in I}\mathcal{B}\left( U_{i}\right)
\rightrightarrows \dprod\limits_{i,j\in I}\mathcal{B}\left( U_{i}\cap
U_{j}\right) \right) .
\end{eqnarray*}%
Therefore, $\mathcal{B}$ is a sheaf. To prove that $\mathcal{C}^{\#}%
\simeq%
\mathcal{B}$, it is enough, due to Theorem \ref{Main-local-iso-pro-K} and
Proposition \ref{Prop-Strong-local-equivalence}, to prove that $\mathcal{%
C\rightarrow B}$ is a $\mathbf{Set}$-local isomorphism of presheaves. The
stalks $\mathcal{C}_{x}=Hom_{\mathbf{K}}\left( G,H\right) $ are constant.
Let $x\in X$, and let $J\left( x\right) $ be the set of open neighborhoods
of $x$. Since%
\begin{equation*}
\mathcal{B}_{x}=\underrightarrow{\lim }_{U\in J\left( x\right) }\mathcal{B}%
\left( U\right) 
\simeq%
\underrightarrow{\lim }_{U\in J\left( x\right) }Hom_{\mathbf{K}}\left(
G,H\right) ^{U}%
\simeq%
Hom_{\mathbf{K}}\left( G,H\right) 
\simeq%
\mathcal{C}_{x},
\end{equation*}%
the morphism $\mathcal{C\rightarrow B}$ is indeed a local isomorphism.

\item If $\mathbf{K}=\mathbf{Set}$, and $G\in \mathbf{Set}$, then $G\otimes
_{\mathbf{Set}}pro$-$\pi _{0}%
\simeq%
G\times pro$-$\pi _{0}$.

\item If $\mathbf{K}=\mathbf{Ab}$, and $G\in \mathbf{Ab}$, then $G\otimes _{%
\mathbf{Set}}pro$-$\pi _{0}%
\simeq%
pro$-$H_{0}\left( \_,G\right) $. Indeed, let $U\rightarrow \left(
Y_{j}\right) _{j\in \mathbf{J}}$ be a polyhedral expansion. Since the
polyhedra $Y_{j}$ are locally connected,%
\begin{equation*}
H_{0}\left( Y_{j},G\right) 
\simeq%
Hom_{\mathbf{Set}}\left( \pi _{0}Y_{j},G\right) 
\simeq%
\dcoprod\limits_{\pi _{0}Y_{j}}G%
\simeq%
G\otimes _{\mathbf{Set}}\pi _{0}\left( Y_{j}\right) .
\end{equation*}
\end{enumerate}
\end{proof}

\bibliographystyle{alpha}
\bibliography{Cosheaves}

\begin{thebibliography}{SGA72}

\bibitem[AM86]{Artin-Mazur-MR883959}
M.~Artin and B.~Mazur.
\newblock {\em Etale homotopy}, volume 100 of {\em Lecture Notes in
  Mathematics}.
\newblock Springer-Verlag, Berlin, 1986.
\newblock Reprint of the 1969 original.

\bibitem[AR94]{Adamek-Rosicky-1994-Locally-presentable-categories-MR1294136}
Ji{\v{r}}{\'{\i}} Ad{\'a}mek and Ji{\v{r}}{\'{\i}} Rosick{\'y}.
\newblock {\em Locally presentable and accessible categories}, volume 189 of
  {\em London Mathematical Society Lecture Note Series}.
\newblock Cambridge University Press, Cambridge, 1994.

\bibitem[Art62]{Artin-GT}
M.~Artin.
\newblock {\em Grothendieck topologies. Notes on a Seminar held in Spring
  1962}.
\newblock Harvard University, Cambridge, MA, 1962.

\bibitem[Bre68]{Bredon-MR0226631}
Glen~E. Bredon.
\newblock Cosheaves and homology.
\newblock {\em Pacific J. Math.}, 25:1--32, 1968.

\bibitem[Bre97]{Bredon-Book-MR1481706}
Glen~E. Bredon.
\newblock {\em Sheaf theory}, volume 170 of {\em Graduate Texts in
  Mathematics}.
\newblock Springer-Verlag, New York, second edition, 1997.

\bibitem[BS84]{Brown-Salleh-1984-A-van-Kampen-theorem-for-unions-on-nonconnect%
ed-spaces-MR751476}
Ronald Brown and Abdul~Razak Salleh.
\newblock A van {K}ampen theorem for unions on nonconnected spaces.
\newblock {\em Arch. Math. (Basel)}, 42(1):85--88, 1984.

\bibitem[Fun95]{Funk-1995-The-display-locale-of-a-cosheaf-MR1322801}
J.~Funk.
\newblock The display locale of a cosheaf.
\newblock {\em Cahiers Topologie G\'eom. Diff\'erentielle Cat\'eg.},
  36(1):53--93, 1995.

\bibitem[KS06]{Kashiwara-Categories-MR2182076}
Masaki Kashiwara and Pierre Schapira.
\newblock {\em Categories and sheaves}, volume 332 of {\em Grundlehren der
  Mathematischen Wissenschaften [Fundamental Principles of Mathematical
  Sciences]}.
\newblock Springer-Verlag, Berlin, 2006.

\bibitem[Mar00]{Mardesic-MR1740831}
Sibe Marde{\v{s}}i{\'c}.
\newblock {\em Strong shape and homology}.
\newblock Springer Monographs in Mathematics. Springer-Verlag, Berlin, 2000.

\bibitem[ML98]{Mac-Lane-Categories-1998-MR1712872}
Saunders Mac~Lane.
\newblock {\em Categories for the working mathematician}, volume~5 of {\em
  Graduate Texts in Mathematics}.
\newblock Springer-Verlag, New York, second edition, 1998.

\bibitem[MS82]{Mardesic-Segal-MR676973}
Sibe Marde{\v{s}}i{\'c} and Jack Segal.
\newblock {\em Shape theory}, volume~26 of {\em North-Holland Mathematical
  Library}.
\newblock North-Holland Publishing Co., Amsterdam, 1982.

\bibitem[Pra12]{Prasolov-smooth-cosheaves-MR2879363}
Andrei~V. Prasolov.
\newblock Precosheaves of pro-sets and abelian pro-groups are smooth.
\newblock {\em Topology Appl.}, 159(5):1339--1356, 2012.

\bibitem[Pra13]{Prasolov-universal-coefficients-formula-2013-MR3095217}
Andrei~V. Prasolov.
\newblock On the universal coefficients formula for shape homology.
\newblock {\em Topology Appl.}, 160(14):1918--1956, 2013.

\bibitem[Sch87]{Schneiders-MR885939}
Jean-Pierre Schneiders.
\newblock Cosheaves homology.
\newblock {\em Bull. Soc. Math. Belg. S\'er. B}, 39(1):1--31, 1987.

\bibitem[SGA72]{SGA4-1-MR0354652}
{\em Th\'eorie des topos et cohomologie \'etale des sch\'emas. {T}ome 1:
  {T}h\'eorie des topos}.
\newblock Lecture Notes in Mathematics, Vol. 269. Springer-Verlag, Berlin,
  1972.
\newblock S{\'e}minaire de G{\'e}om{\'e}trie Alg{\'e}brique du Bois-Marie
  1963--1964 (SGA 4), Dirig{\'e} par M. Artin, A. Grothendieck, et J. L.
  Verdier. Avec la collaboration de N. Bourbaki, P. Deligne et B. Saint-Donat.

\bibitem[Sug01]{Sugiki-2001-33}
Yuichi Sugiki.
\newblock {\em The category of cosheaves and {L}aplace transforms}, volume
  2001-33 of {\em UTMS Preprint Series}.
\newblock University of Tokyo, Tokyo, 2001.

\bibitem[Tam94]{Tamme-MR1317816}
G{\"u}nter Tamme.
\newblock {\em Introduction to \'etale cohomology}.
\newblock Universitext. Springer-Verlag, Berlin, 1994.
\newblock Translated from the German by Manfred Kolster.

\bibitem[Woo09]{Woolf-2009-MR2591969}
Jon Woolf.
\newblock The fundamental category of a stratified space.
\newblock {\em J. Homotopy Relat. Struct.}, 4(1):359--387, 2009.

\bibitem[Woo15]{Woolf-2015-Erratum-to-The-fundamental-category-of-a-stratified%
-space-MR3313639}
Jon Woolf.
\newblock Erratum to: {T}he fundamental category of a stratified space.
\newblock {\em J. Homotopy Relat. Struct.}, 10(1):123--125, 2015.

\end{thebibliography}

\end{document}